\documentclass{amsart}
\usepackage{amsmath}
\usepackage{paralist}
\usepackage{amsfonts}
\usepackage{amssymb}
\usepackage{amsthm}
\usepackage{amscd}
\usepackage{amsrefs}
\usepackage{float}
\usepackage{tikz}
\usepackage{graphicx}
\usepackage[colorlinks=true]{hyperref}
\hypersetup{urlcolor=blue, citecolor=red}
\usepackage{hyperref}

\newtheorem{theorem}{Theorem}[section]

\newtheorem{lemma}[theorem]{Lemma}
\newtheorem{proposition}[theorem]{Proposition}

\theoremstyle{definition}
\newtheorem{definition}[theorem]{Definition}
\newtheorem{remark}[theorem]{Remark}

\newcommand{\T}{\mathbb{T}}
\newcommand{\R}{\mathbb{R}}
\newcommand{\Z}{\mathbb{Z}}

\begin{document}

\title[Shorttime bilinear Strichartz estimates] 
      {On shorttime bilinear Strichartz estimates and applications to improve the energy method}

\author[R. Schippa]{Robert Schippa}
\address{Fakult\"at f\"ur Mathematik, Universit\"at Bielefeld, Postfach 10 01 31, 33501 Bielefeld, Germany}
\email{robert.schippa@uni-bielefeld.de}

\thanks{Financial support by the German Science Foundation (IRTG 2235) is gratefully acknowledged.}
\maketitle

\bigskip

\begin{abstract}
A refinement of the Bona-Smith method is introduced for dispersive PDE with derivative nonlinearity posed on tori. Key ingredient is a shorttime bilinear Strichartz estimate, which is used in a known combination of perturbative and energy arguments.
\end{abstract}

\section{Introduction}
We consider classical models with derivative nonlinearity like the $k$-generalized Benjamin-Ono equation giving rise to periodic solutions
\begin{equation}
\label{eq:BenjaminOnoEquation}
\left\{ \begin{array}{cl}
\partial_t u + \mathcal{H} \partial_{xx} u &= u^{k-1} \partial_x u \quad (t,x) \in \mathbb{R} \times \mathbb{T} \\
u(0) &= u_0 \in H^s(\mathbb{T}) \end{array} \right.
\end{equation}
or the Zakharov-Kuznetsov equation
\begin{equation}
\label{eq:ZakharovKuznetsovEquation}
\left\{ \begin{array}{cl}
\partial_t u + \partial_{xxx} u + 3 \partial_x \partial_{yy} u &= u \partial_x u \quad (t,x) \in \mathbb{R} \times \mathbb{T}^2 \\
u(0) &= u_0 \in H^s(\mathbb{T}^2) \end{array} \right.
\end{equation}
In \eqref{eq:BenjaminOnoEquation} $\mathcal{H}$ denotes the Hilbert transform, i.e.,
\begin{equation*}
\mathcal{H}: L^2(\mathbb{T}) \rightarrow L^2(\mathbb{T}), \quad (\mathcal{H} f) \widehat (\xi) = - i sgn(\xi) \hat{f}(\xi)
\end{equation*}
In \eqref{eq:ZakharovKuznetsovEquation} the constants in front of the spatial derivatives are relevant for the argument of the proof. Fixing the constants is equivalent to prescribing a ratio of the period lengths.\\
In case $k=2$ \eqref{eq:BenjaminOnoEquation} becomes the Benjamin-Ono equation.
Note that in both cases the mean value is a conserved quantity and real initial values give rise to real-valued solutions. Hence, we confine ourselves in the following to real-valued initial data with vanishing mean. The relevant subspace of the Sobolev space $H^s$ will be denoted by $H_{\mathbb{R}}^s$. For expositional purposes we confine ourselves for most of the time to quadratic nonlinearities, i.e., $k=2$. Later, we shall see how the results generalize to higher order nonlinearities for both models.\\
In both cases the derivative loss prevents one from concluding well-posedness by perturbative methods (cf. \cite{KochTzvetkov2005,MolinetSautTzvetkov2001,Herr2008}).\\
Computing the change of the $L^2$-norm of the $N$ frequencies of solutions to \eqref{eq:BenjaminOnoEquation} for $k=2$ or \eqref{eq:ZakharovKuznetsovEquation} we find
\begin{equation*}
\Vert P_N u(t) \Vert^2_{L^2} = \Vert P_N u(0) \Vert_{L^2}^2 + 2 \int_0^t ds \int_{\mathbb{T}} dx P_N u P_N(u \partial_x u)
\end{equation*}
and the most problematic $High \times Low \rightarrow High$-interaction $P_N(u P_{<N} u)$ can be handled with integration by parts and a commutator estimate. One can conclude the argument with Sobolev embedding leading to well-posedness for $s \geq 1 + d/2$, where $d$ denotes the spatial dimensions. This is one of the main arguments of the classical Bona-Smith method (cf. \cite{BonaSmith1975}).\\
However, at none of the steps the dispersive properties of the propagator are taken into account.\\
Of the above equations the Benjamin-Ono equation is best understood. The first well-posedness result below $s=3/2$ on the real line was reached by Koch-Tzvetkov in \cite{KochTzvetkov2003} using improved Strichartz estimates considering small frequency-dependent time intervals.\\
In this work we will use shorttime bilinear Strichartz estimates.\\
A breakthrough in the well-posedness theory of the Benjamin-Ono equation was the proof of global well-posedness in $H^1(\mathbb{R})$ by Tao in \cite{Tao2004}, where a gauge transform mimicking the Cole-Hopf transform was used weakening the derivative loss enough to treat the equation with Strichartz estimates.\\
This strategy was further refined and Ionescu-Kenig \cite{IonescuKenig2007} proved global well-posedness in $L^2(\mathbb{R})$, which is the best result so far although the scaling critical regularity is $s_c = -1/2$.\\
Adapting the gauge transform to the periodic setting global well-posedness was proved in $L^2(\mathbb{T})$ by Molinet in \cite{Molinet2008}. The proofs of global well-posedness in $L^2$ were simplified in \cite{MolinetPilod2012,IfrimTataru2017}.\\
However, all these results heavily draw from the gauge transform which is very accessible for the Benjamin-Ono equation. For the Zakharov-Kuznetsov equation it is currently unknown whether there is a suitable gauge transform. Using shorttime linear Strichartz estimates, very recently local well-posedness of \eqref{eq:ZakharovKuznetsovEquation} was proved for $s>5/3$ in \cite{LinaresPantheeRobertTzvetkov2018} and in fact, for any period lengths. On $\mathbb{R}^2$ the Zakharov-Kuznetsov equation is known to be semilinear and locally well-posed for $s>1/2$ (cf. \cite{MolinetPilod2015,GruenrockHerr2014}). Making use of shorttime bilinear estimates we modestly improve the local well-posedness of \eqref{eq:ZakharovKuznetsovEquation} to $s>3/2$. On $\mathbb{R}^2$ the Zakharov-Kuznetsov equation is equivalent to
\begin{equation}
\label{eq:SymmetrizedZakharovKuznetsovEquation}
\left\{ \begin{array}{cl}
\partial_t u + \partial_{xxx} u + \partial_{yyy} u &= u (\partial_x + \partial_y) u, \quad (t,x) \in \R \times \R^2\\
u(0) &= u_0 \in H^s(\mathbb{R}^2) \end{array} \right.
\end{equation}
This is no longer the case on $\mathbb{T}^2$. However, the method of proof yields that \eqref{eq:SymmetrizedZakharovKuznetsovEquation} is also locally well-posed on $\lambda_1 \T \times \lambda_2 \T$ provided that $s>3/2$.\\
The improvement of the energy method by Molinet-Vento (cf. \cite{MolinetVento2015}) yields well-posedness of the Benjamin-Ono equation in $H^{1/2}(\mathbb{K})$, $\mathbb{K} \in \{ \mathbb{R}, \mathbb{T} \}$. The improvement hinges on the analysis of the resonance function, which is less clear for the generalized Benjamin-Ono equation or the Zakharov-Kuznetsov equation. The method we will describe in the present work is more flexible, in the sense that the nonlinear and energy estimates rely mostly on shorttime bilinear Strichartz estimates, so that it can also be used for equations which are not amenable to gauge transforms or have a more involved resonance.
  Moreover, this approach can also be applied to solutions defined on the Euclidean space, and likely even on cylinders. We choose to illustrate the method on tori because on compact manifolds the dispersive properties are not as clear as in Euclidean space, where the dispersive estimate
\begin{equation}
\label{eq:dispersiveEstimate}
\Vert e^{it \Delta} u_0 \Vert_{L^\infty(\mathbb{R}^n)} \lesssim_n |t|^{-n/2} \Vert u_0 \Vert_{L^1(\mathbb{R}^n)}
\end{equation}
holds globally in time. On compact manifolds this would contradict the conservation of mass.\\
In now already classical works it was demonstrated how localizing time to small frequency dependent time intervals recovers the dispersive properties from Euclidean space (cf. \cite{StaffilaniTataru2002, BurqGerardTzvetkov2004}).\\
To be more precise, consider a smooth dispersion relation $\varphi: \mathbb{R}^n \rightarrow \mathbb{R}$. The underlying heuristic is that a frequency localized solution with frequencies around $\xi$ to the linear dispersive PDE
\begin{equation}
i \partial_t u + \varphi(\nabla/i) u = 0
\end{equation}
travels with a group velocity $|\nabla \varphi(\xi)|$ as can be seen from writing the linear propagator in Fourier space.\\
For the refinement of the energy method we plan to use bilinear Strichartz estimates: In case of Schr\"odinger's equation we have $\varphi(\xi) = -\xi^2$ that means Schr\"odinger waves even on compact manifolds do not leave one chart in a time interval of size $N^{-1}$. This time localization we refer to as Euclidean window. This nomenclature was previously used in a different context in \cite{IonescuPausader2012}. This is the time interval in which linear Strichartz estimates can be recovered (cf. \cite{BurqGerardTzvetkov2004}).
To quantify the speed of propagation we say that $\varphi$ is of order $\alpha>1$ when
\begin{equation}
\label{eq:alphaGroupVelocity}
|\nabla \varphi(\xi)| \sim N^{\alpha-1} \quad N \in 2^{\mathbb{N}_0}
\end{equation}
 Multilinear estimates can be improved making additional use of transversality. The following result in Euclidean space is well-known:
\begin{proposition}
Let $U_i$ be open sets in $\mathbb{R}^n$, $\varphi_i \in C^1(U_i,\mathbb{R})$ and let $u_i$ have Fourier support in balls of radius $r$ which are contained in $U_i$ for $i=1,2$. Moreover, suppose that $|\nabla \varphi(\xi_1) - \nabla \varphi(\xi_2)| \geq N > 0 $, whenever $\xi_1 \in U_1$, $\xi_2 \in U_2$.\\
Then, we find the following estimate to hold:
\begin{equation}
\label{eq:EuclideanTransversality}
\Vert e^{it \varphi(\nabla/i)} u_1 e^{it \varphi(\nabla/i)} u_2 \Vert_{L^2_{t,x}(\mathbb{R} \times \mathbb{R}^n)} \lesssim_n \frac{r^{\frac{n-1}{2}}}{N^{1/2}}  \Vert u_1 \Vert_{L^2(\mathbb{R}^n)} \Vert u_2 \Vert_{L^2(\mathbb{R}^n)}
\end{equation}
\end{proposition}
In Euclidean space this follows from a change of variables. A consequence of this transversality are bilinear Strichartz estimates for Schr\"odinger waves localized at frequencies $N_1 \gg N_2$ (cf. \cite{Bourgain1998}):
\begin{equation}
\label{eq:bilinearStrichartzEstimateSEQ}
\Vert e^{it \Delta} P_{N_1} u_1 e^{it \Delta} P_{N_2} u_2 \Vert_{L^2_{t,x}(\mathbb{R} \times \mathbb{R}^n)} \lesssim \frac{N_2^{\frac{n-1}{2}}}{N_1^{1/2}} \Vert P_{N_1} u_1 \Vert_{L^2} \Vert P_{N_2} u_2 \Vert_{L^2}
\end{equation}
For the Schr\"odinger equation this result was proved in Euclidean windows by Hani (cf. \cite{Hani2012}) on general compact manifolds. On $\mathbb{T}$ the estimate, which is to be expected working in Euclidean windows, was proven in \cite{MoyuaVega2008}. We prove a generalization of this estimate on higher dimensional tori in Euclidean windows.
However, even for solutions on the Euclidean space this estimate is not sufficient to overcome the derivative loss in case of the Benjamin-Ono equation, so that the equation can be treated as seminlinear. Using Duhamel's formula we write
\begin{equation}
u(t) = e^{t \mathcal{H} \partial_{xx}} u_0 + \int_0^t e^{(t-s) \mathcal{H} \partial_{xx} } (u \partial_x u)(s) ds 
\end{equation}
Now consider the $High \times Low \rightarrow High$-interaction in which case the derivative loss must be completely recovered in order to close the argument:
\begin{equation}
\begin{split}
&\Vert \int_0^T e^{(t-s) \mathcal{H} \partial_{xx}} \partial_x \tilde{P}_N (P_N u P_K u) ds \Vert_{L_x^2} \\
&\lesssim \Vert \partial_x (P_N u P_K u) \Vert_{L_{T}^1 L_x^2} \lesssim |T|^{1/2} N \Vert P_N u P_K u \Vert_{L_T^2 L_x^2} \\
&\lesssim |T|^{1/2} N^{1/2} \Vert P_N u(0) \Vert_{L^2} \Vert P_K u(0) \Vert_{L^2}
\end{split}
\end{equation}
and in order to compensate the derivative loss one has to choose $T=T(N) = N^{-1}$. This observation was made precise in \cite{GuoPengWangWang2011}, where global well-posedness was proven in $H^1(\mathbb{R})$ without gauge transform. One key observation is that at this localization in time the bilinear Strichartz estimate also holds on the torus. To exploit this, we utilize function spaces incorporating the frequency dependent localization in time in order to prove the following theorems:
\begin{theorem}
\label{thm:wellPosednessBenjaminOnoEquation}
Let $k \in \mathbb{Z}_{\geq 2}$ and $s>1$.
\begin{enumerate}
\item[(a)] There exists $\varepsilon = \varepsilon(k,s)$ so that for any solution $u$ to \eqref{eq:BenjaminOnoEquation} we find the following estimate to hold
\begin{equation}
\label{eq:APrioriEstimateSmoothSolutions}
\sup_{t \in [0,1]} \Vert u(t) \Vert_{H_{\mathbb{R}}^s} \lesssim \Vert u_0 \Vert_{H_{\mathbb{R}}^s}
\end{equation}
provided that $u(0)$ is a smooth real-valued initial datum with vanishing mean and $\Vert u(0) \Vert_{H_{\mathbb{R}}^s} \leq \varepsilon$.
\item[(b)] There exists $\varepsilon = \varepsilon(k,s)$ so that for solutions $u_1, u_2$ we find the following estimate to hold:
\begin{equation}
\sup_{t \in [0,1]} \Vert u_1(t) - u_2(t) \Vert_{L^2} \lesssim_{\Vert u_i(0) \Vert_{H^s}} \Vert u_1(0) - u_2(0) \Vert_{L^2}
\end{equation}
provided that $u_i(0)$ is a smooth real-valued initial datum with vanishing mean and $\Vert u_i(0) \Vert_{H^s} \leq \varepsilon(k,s)$ for $i=1,2$.
\item[(c)] For any $T>0$ there is an $\varepsilon > 0$ so that the data-to-solution mapping $S_T^\infty: H_{\mathbb{R}}^\infty \rightarrow C([0,T],H_{\mathbb{R}}^\infty)$ admits a unique continuous extension $S_T^s: B_\varepsilon \rightarrow C([0,T],H_{\mathbb{R}}^s)$ where $B_\varepsilon$ denotes the $\varepsilon$ ball around the origin in $H_{\mathbb{R}}^s$.
\end{enumerate}
\end{theorem}
This result recovers the results from \cite{MolinetRibaud2009} on generalized Benjamin-Ono equations on the circle up to $s=1$ for small initial data. In \cite{MolinetRibaud2009} the well-posedness result was established by means of a gauge transform recasting the derivative nonlinearity into a milder form. From the proof it will be clear that we can also treat linear combinations of derivative nonlinearities $\sum_{k \geq 2}^K a_k u^{k-1} \partial_x u$ which is not possible when one uses a gauge transform because the gauge transform changes with the order of nonlinearity. In case of mixed orders we have to consider additional smallness conditions on the initial data. Moreover, without using a gauge transform it is straightforward to analyze viscous limits (cf. \cite{Molinet2013, GuoPengWangWang2011}). The requirement $s>1$ only stems from a logarithmic loss of carrying out square sums. Switching to $\ell^1$-Besov refinements we can also recover $s=1$.\\
For the two-dimensional Zakharov-Kuznetsov equation we can prove the following theorem by the same means:
\begin{theorem}
\label{thm:wellPosednessZakharovKuznetsov}
Let $T > 0$ and $s > 3/2$. There exists $\varepsilon = \varepsilon(T,s)$ so that the mapping $S_T^\infty$ assigning smooth real-valued intial data with vanishing mean to smooth solutions to \eqref{eq:ZakharovKuznetsovEquation} has a unique continuous extension $S_T^s: B_\varepsilon \rightarrow C([0,T],H_{\mathbb{R}}^s)$.
\end{theorem}
The article is structured as follows: In Section \ref{section:ShorttimeBilinearStrichartzEstimates} we prove the bilinear Strichartz estimates in Euclidean windows on tori. In Section \ref{section:FunctionSpaces} we define function spaces which take into account frequency dependent time localization. In Section \ref{section:ShorttimeNonlinearEstimates} we iterate the quadratic nonlinearity in the shorttime function spaces and in Section \ref{section:EnergyEstimates} we perform energy estimates. In Section \ref{section:Conclusion} we conclude the argument to prove Theorems \ref{thm:wellPosednessBenjaminOnoEquation} and \ref{thm:wellPosednessZakharovKuznetsov} for $k=2$. In Section \ref{section:HigherOrderGeneralizations} we indicate how to extend the result to the cases $k \geq 3$.

\section{Shorttime bilinear Strichartz estimates}
\label{section:ShorttimeBilinearStrichartzEstimates}
Purpose of this section is to prove shorttime bilinear Strichartz estimates which resemble the bilinear Strichartz estimates one has due to transversality in Euclidean space. We are not able to recover the full result and it is unclear whether these estimates hold true in the same generality like in Euclidean space. We have to impose a condition allowing us to disentagle the oscillation in different coordinates.

\begin{definition}
We say that a dispersion relation $\varphi: \mathbb{R}^n \rightarrow \mathbb{R}$ is of sum type if $\varphi(\xi) = \sum_{i=1}^n \mu(\xi_i)$ with $\mu$ slowly varying, i.e., $\mu(x) \sim \mu(2x)$ for any $x \neq 0$.
\end{definition}

\begin{proposition}
\label{prop:shorttimeBilinearStrichartzEstimate}
Let $K \ll N$, suppose that $\varphi$ is of sum type and satisfies \eqref{eq:alphaGroupVelocity} for some $\alpha > 1$. Then, we find the following estimate to hold:
\begin{equation}
\label{eq:bilinearStrichartzEstimate}
\begin{split}
&\Vert P_N e^{it \varphi(\nabla/i)} u_1 P_K e^{it \varphi(\nabla/i)} u_2 \Vert_{L_t^2([0,N^{-(\alpha-1)}],L^2(\mathbb{T}^n))} \\
 &\lesssim \frac{K^{\frac{n-1}{2}}}{N^{\frac{\alpha-1}{2}}} \Vert P_N u_1 \Vert_{L^2(\mathbb{T}^n)} \Vert P_K u_2 \Vert_{L^2(\mathbb{T}^n)}
\end{split}
\end{equation}
\end{proposition}
\begin{proof}
We find
\begin{equation*}
\begin{split}
u_1(t) &= \sum_{k_1 \in \mathbb{Z}^n} e^{i k_1. x} e^{it \varphi(k_1)} a(k_1), \quad u_2(t) = \sum_{k_2 \in \mathbb{Z}^n} e^{i k_2. x} e^{it \varphi(k_2)} b(k_2) \\
u_1 u_2(t) &= \sum_{k_1,k_2 \in \mathbb{Z}^n} e^{i(k_1+k_2).x} [ e^{it [ \varphi(k_1) + \varphi(k_2)]} a(k_1) b(k_2) ]
\end{split}
\end{equation*}
Consequently, Plancherel's theorem yields
\begin{equation}
\label{eq:L2Norm}
\begin{split}
\Vert u_1 u_2 \Vert^2_{L^2} &= \sum_{k \in \mathbb{Z}^n} \left| \sum_{k_2 \in \mathbb{Z}^n} e^{it( \varphi(k-k_2) + \varphi(k_2)} a(k-k_2) b(k_2) \right|^2 \\
&= \sum_{k \in \mathbb{Z}^n} \sum_{k_2^{(1)}, k_2^{(2)} \in \mathbb{Z}^n} e^{it([\varphi(k-k_2^{(1)}) + \varphi(k_2^{(1)})] - [ \varphi(k-k_2^{(2)}) + \varphi(k_2^{(2)})])} a(k-k_2^{(1)}) b(k_2^{(1)}) \\
&\times \overline{a(k-k_2^{(2)})} \overline{b(k_2^{(2)})}
\end{split}
\end{equation}
Set $\psi_k(k^\prime) = \varphi(k-k^\prime) + \varphi(k^\prime)$. Next, let $\eta_\delta(t) = \eta(t/\delta)$ where $\eta$ is a suitable bump function and majorize
\begin{equation*}
\int_0^{N^{-(\alpha-1)}} dt \Vert u_1 u_2(t) \Vert_{L^2(\mathbb{T}^n)}^2 \leq \int \eta_\delta(t) \Vert u_1 u_2(t) \Vert^2_{L^2(\mathbb{T}^n)}, \quad \delta = N^{-(\alpha-1)}
\end{equation*}
and we find
\begin{equation}
\label{eq:timeIntegralL2Norm}
\begin{split}
\int \eta_\delta(t) \eqref{eq:L2Norm}(t) dt &= \sum_{k \in \mathbb{Z}^n} \sum_{k_2^{(1)}, k_2^{(2)} \in \mathbb{Z}^n} \hat{\eta}_\delta(\psi_k(k_2^{(1)})-\psi_k(k_2^{(2)})) \\
&\times a(k-k_2^{(1)}) b(k_2^{(1)}) \overline{a(k-k_2^{(2)})} \overline{b(k_2^{(2)})}
\end{split}
\end{equation}
The inner sum we will estimate with Young's inequality. Note that
\begin{equation}
\label{eq:phaseDifference}
\begin{split}
&\psi_k(k_2^{(1)}) - \psi_k(k_2^{(2)}) \\
&= \int_0^1 \nabla \psi_k(k_2^{(2)} + t(k_2^{(1)} - k_2^{(2)})) (k_2^{(1)} - k_2^{(2)}) dt \\
&= \int_0^1 [\nabla \varphi(k_2^{(2)} + t(k_2^{(1)} - k_2^{(2)})) - \nabla \varphi(k - (k_2^{(2)} + t(k_2^{(1)}-k_2^{(2)})))] dt (k_2^{(1)} - k_2^{(2)})
\end{split}
\end{equation}
By assumption it is easy to see that there is one component of the integral which is of order $N^{\alpha-1}$ independent of $t$, say the first component. This gives
\begin{equation*}
\eqref{eq:phaseDifference} = (N^{\alpha-1} c_1(k_1,k_{21}^{(1)},k_{21}^{(2)})) (k_{21}^{(1)}-k_{21}^{(2)}) + \sum_{i=2}^n C_i(k_i,k_{2i}^{(1)},k_{2i}^{(2)})(k_{2i}^{(1)}-k_{2i}^{(2)}),
\end{equation*}
where due to our assumptions on $\mu$ there is $C>0$ so that
\begin{equation*}
C^{-1} \leq \pm c_1(k_1,k_{21}^{(1)},k_{21}^{(2)}) \leq C
\end{equation*}
An application of Young's inequality yields
\begin{equation*}
\begin{split}
\eqref{eq:timeIntegralL2Norm} &\lesssim \sum_{k \in \mathbb{Z}^n} \{ \sup_{k_2^{(2)} \in \mathbb{Z}^n} \sum_{k_2^{(1)} \in \mathbb{Z}^n} | \delta \hat{\eta}((\delta N^{\alpha-1} c_1(k_1,k_{21}^{(1)},k_{21}^{(2)})(k_{21}^{(1)}-k_{21}^{(2)}) \\
&+ \sum_{i=2}^n C_i(k_i,k_{2i}^{(1)},k_{2i}^{(2)})(k_{2i}^{(1)} - k_{2i}^{(2)}))| \} \times \sum_{k_2 \in \mathbb{Z}^n} |a(k-k_2) b(k_2)|^2
\end{split}
\end{equation*}
The sum $\sum_{k_{21}^{(1)} \in \mathbb{Z}} | \hat{\eta} (\ldots)|$ is majorized by $\int |\hat{\eta}(\xi)| d\xi$ and summing over the remaining indices yield a factor $K$ per summation.\\
Consequently,
\begin{equation*}
\begin{split}
\eqref{eq:timeIntegralL2Norm} &\lesssim \sum_{k \in \mathbb{Z}^n} \delta K^{n-1} \sum_{k_2 \in \mathbb{Z}^n} |a(k-k_2)|^2 |b(k_2)|^2 \\
&\lesssim \delta K^{n-1} \Vert a \Vert^2_{2} \Vert b \Vert^2_{2}
\end{split}
\end{equation*}
and the proof is complete.
\end{proof}
Observe how the special form of $\varphi$ comes into play in the expression \eqref{eq:phaseDifference} and the subsequent estimates. Although the intuition of Euclidean windows predicts Proposition \ref{prop:shorttimeBilinearStrichartzEstimate} to be true in greater generality, from the proof it is unclear how to extend the result to more general phase functions.

In the one-dimensional case this estimate was proved up to complex conjugation in \cite[Theorem~4,~p.~125]{MoyuaVega2008}.
We have a look at examples: In the one-dimensional case one can consider the equations:
\begin{equation*}
i \partial_t u + D^a u = 0, \quad a > 1, \quad u: \mathbb{R} \times \mathbb{T} \rightarrow \mathbb{C}
\end{equation*}
or, similarly,
\begin{equation*}
\partial_t u + \partial_x D_x^{a-1} u = 0, \quad u: \mathbb{R} \times \mathbb{T} \rightarrow \mathbb{R}
\end{equation*}
In both cases Proposition \ref{prop:shorttimeBilinearStrichartzEstimate} yields
\begin{equation*}
\Vert P_N u_1(t) P_K u_2(t) \Vert_{L_t^2([0,N^{-(\alpha-1)}],L^2(\mathbb{T}^n)} \lesssim N^{-\frac{\alpha-1}{2}} \Vert P_N u_1(0) \Vert_{L^2(\mathbb{T}^n)} \Vert P_K u_2(0) \Vert_{L^2(\mathbb{T}^n)}
\end{equation*}
In higher dimensions one can still consider
\begin{equation*}
i \partial_t u + \Delta u = 0, \quad u: \mathbb{R} \times \mathbb{T}^n \rightarrow \mathbb{C}
\end{equation*}
In this case Proposition \ref{prop:shorttimeBilinearStrichartzEstimate} recovers the result \cite[Theorem~1.2,~p.~343]{Hani2012} in the special case of $M = \mathbb{T}^{n}$.\\
Later, we will deal with derivative nonlinearities and in order to integrate by parts when carrying out energy estimates one had to consider the artifical nonlinearity $u \partial_x \overline{u}$, which can be treated without usage of shorttime norms. In two dimensions one has the more interesting example of the linear part of the Zakharov-Kuznetsov equation:
\begin{equation*}
\partial_t u + \partial_x^3 u + 3 \partial_x \partial_y^2 u = 0, \quad (x,t) \in \mathbb{T}^2 \times \mathbb{R}
\end{equation*}
Proposition \ref{prop:shorttimeBilinearStrichartzEstimate} does not apply directly because the dispersion relation $\varphi(k_1,k_2) = k_1^3 + 3 k_1 k_2^2$ does not separate time oscillations of the coordinates. However, performing a transformation symmetrizing the Zakharov-Kuznetsov equation (cf. \cite{BenArtziKochSaut2003,GruenrockHerr2014}) in momentum space we find
\begin{equation*}
\begin{split}
\varphi(A(k_1^\prime,k_2^\prime)) &= (k_1^\prime)^3 + (k_2^\prime)^3, \text{ where} \\
A(k_1^\prime,k_2^\prime) &= \mu ((k_1^\prime+k_2^\prime),(k_1^\prime-k_2^\prime)), \; \mu = 4^{-1/3}.
\end{split}
\end{equation*}
The linear transformation is essentially a rotation. Here, the fact that $A$ is up to a constant a matrix with integer coefficients becomes important in the proof of bilinear Strichartz estimates: This will make $A^{-1}(\Z^2)$ again a grid and allows us to carry out the argument in the subsequent estimate of \eqref{eq:timeIntegralL2Norm} in the rotated coordinates.\\
When we compute the $L^2$-norm of a product of free solutions we arrive at the expression
\begin{equation*}
\sum_{k \in \mathbb{Z}^n} \sum_{\substack{k_1^{(1)}+k_2^{(1)}=k,\\ k_1^{(2)}+k_2^{(2)}=k}} e^{it [ (\varphi(k_1^{(1)})+\varphi(k_2^{(1)})) - (\varphi(k_1^{(2)}) + \varphi(k_2^{(2)}))]} a(k_1^{(1)}) b(k_2^{(1)}) \overline{a(k_1^{(2)}) b(k_2^{2})}
\end{equation*}
and after changing variables $k = Am, \; k_i^{(j)} = A m_i^{(j)}$ we find
\begin{equation*}
\begin{split}
&\sum_{m \in A^{-1}(\mathbb{Z}^2)} \sum_{\substack{m_1^{(1)} + m_2^{(1)} = m, \\ m_1^{(2)} + m_2^{(2)}=m}} e^{it[( \chi(m_1^{(1)}) + \chi(m_2^{(1)})) - (\chi(m_1^{(2)}) + \chi(m_2^{(2)}))]} a(A m_1^{(1)}) b(A m_2^{(1)}) \\
&\times \overline{a(A m_1^{(2)})} \overline{b(A m_2^{(2)})},
\end{split}
\end{equation*}
where $\chi(k_1,k_2) = k_1^3 + k_2^3$ is of sum type.\\
Using this transformation one proves like above
\begin{equation*}
\begin{split}
&\Vert P_N e^{it \varphi_{ZK}(\nabla/i)} u_1(0) P_K e^{it \varphi_{ZK}(\nabla/i)} u_2(0) \Vert_{L_t^2([0,N^{-2}],L^2(\mathbb{T}^n))} \\
&\lesssim \frac{K^{1/2}}{N} \Vert P_N u_1(0) \Vert_{L^2} \Vert P_K u_2(0) \Vert_{L^2}
\end{split}
\end{equation*}
\begin{remark}
\label{remark:EuclideanTransversalityExtension}
We illustrate the argument and some of its consequences.\\
Suppose that $n=1$ and $u_1$ and $u_2$ have Fourier support in intervals $I_1$ and $I_2$, respectively, and for the sake of definiteness consider the dispersion relation $\varphi(\xi) = \xi^3$. Suppose that $I_1$, $I_2$ do not necessarily belong to dyadically separated annuli, but simply satisfy
\begin{equation*}
\nabla \varphi (\xi_1) - \nabla \varphi(\xi_2) \sim \pm N^2, \text{ where } \xi_i \in I_i
\end{equation*}
The Fourier support must be convex so that when we are integrating
\begin{equation*}
\int_0^1 \nabla \varphi(\underbrace{k_2^{(2)}+t(k_2^{(1)}-k_2^{(2)}}_{k^\prime})) - \nabla \varphi(k-k^\prime))dt
\end{equation*}
$k^\prime$ is always an element of $I_2$ and $k-k^\prime$ is always an element of $I_1$ yielding the integral to be $\sim \pm N^{\alpha-1}$. Then, the proof can be carried out along the above lines.
We shall see that we can also deal with $High \times High \times High \times Low \times \ldots$-interaction with two bilinear estimates:
Intuitively, this is the case because the difference of the group velocity is $\xi_1^2 - \xi_2^2 = (\xi_1-\xi_2)(\xi_1+\xi_2)$.\\
Three frequencies $\xi_1,\xi_2,\xi_3$ satisfy $|\xi| \sim N$ with nontrivial interaction. Consequently, $|\xi_1 + \xi_2| \sim N, |\xi_2+\xi_3| \sim N, |\xi_1+\xi_3| \sim N$. Due to nontrivial interaction there has to be one combination so that $|\xi_i-\xi_j| \sim N$ (otherwise $\xi_1+ \xi_2 + \xi_3 + remaining (smaller) frequencies$ do not sum up to zero.) This combination will be amenable to a bilinear Strichartz-estimate.\\
More precisely, in the proof of bilinear Strichartz estimates we come across the integral
\begin{equation}
\int_0^1 dt \nabla \phi(k_2^{(2)} + t(k_2^{(1)} - k_2^{(2)})) - \nabla \phi(k-(k_2^{(2)}+t(k_2^{(1)}-k_2^{(2)}))),
\end{equation}
which has to be of the form $N^2 c(k,k_2^{(1)},k_2^{(2)})$. Then the proof gives the desired estimate.\\
Moreover, it is straightforward to divide the frequency projector into the correct intervals: Write
\begin{equation}
P_N u_1 P_N u_2 P_N u_3 P_K u_4 \ldots = \sum_{I_1,I_2,I_3} P_{I_1} u_1 P_{I_2} u_2 P_{I_3} u_3 P_K u_4 \ldots
\end{equation}
Here, $I_i$ denote intervals of length $cN$, $c \ll 1$. With the intervals being of magnitude $cN$ there will be no loss summing up the different contributions at last. Note that for most of the cases we will just use the separation between $I_1$ and $I_2$, unless these intervals are actually neighbours. Say $I_1$ and $I_2$ are subsets of $\R^{>0}$. Hence, the other frequency will be negative: The mirrored interval of $I_3$ (where $\xi_3$ is contained) with nontrivial interaction must be at least one small interval apart from $I_1$ due to otherwise impossible frequency interaction. In this case the separation between $I_1$ and $I_3$ (and also $-I_3$) will be enough to apply a bilinear Strichartz estimate.\\
Consequently, we record the estimates
\begin{equation*}
\begin{split}
&\Vert S_{>N} (P_{N_1} e^{t \partial_x D_x^{\alpha-1}} u_1(0) P_{N_2} e^{t \partial_x D_x^{\alpha-1}} u_2(0)) \Vert_{L_t^2([0,N^{-(\alpha-1)}],L^2)} \\
&\lesssim N^{-\frac{\alpha-1}{2}} \Vert P_{N_1} u_1(0) \Vert_{L^2} \Vert P_{N_2} u_2(0) \Vert_{L^2}
\end{split}
\end{equation*}
where $S_{>\lambda}$ denotes the part where the modulus of the frequencies is separated of order $\lambda$ and $N_1 \sim N_2 \sim N$. This follows from the interval slicing argument depicted above.\\
Next, we turn to $High \times High \rightarrow High$-interaction in the twodimensional case:
\begin{equation*}
P_N e^{it \varphi(\nabla/i)} u_0 P_{N^\prime} e^{it \varphi(\nabla/i)} v_0
\end{equation*}
Here, we decompose 
\begin{equation*}
P_N = \sum_{I_1,I_2: dyadic} P_N^{I^N_1,I^N_2} + P_{N}^{I_1^N,I_2^{\ll N}} + P_N^{I_1^{\ll N},I_2^N}
\end{equation*}
where $P_N^{I_1,I_2}$ projects to the frequency subset of $A_N$ with first decoupled coordinates in $I_1$, second decoupled coordinates in $I_2$.\\
Suppose that $I_1$ is of magnitude $N$ (otherwise $I_2$ has to be). Then we find for the intervals describing the first decoupled frequency in $P_N^{I_1,I_2} u_2, P_N^{I_1,I_2} u_3$ the following dichotomy: Either one $I_1$ is of much smaller magnitude, then it is amenable to the $High \times Low$-estimate in one dimension, or all three intervals are of order $N$. Then we are in the $High \times High \rightarrow High$-interaction in one dimension and we can argue like above. Summing over dyadic intervals gives an additional factor $\log^6 (N)$.\\
Moreover, rescaling solutions $u(t,x) \rightarrow u(\lambda^k t, \lambda x)$ yields the estimate \eqref{eq:bilinearStrichartzEstimate} with the same constant on a rescaled domain.
\end{remark}
To conclude the section we compare the linear Strichartz estimate from \cite{LinaresPantheeRobertTzvetkov2018} to the bilinear estimate.\\
For \eqref{eq:ZakharovKuznetsovEquation} the linear Strichartz estimate \cite[Equation~(2.4),~p.~4]{LinaresPantheeRobertTzvetkov2018} reads
\begin{equation}
\label{eq:linearZKEstimate}
\Vert e^{it \varphi_{ZK}(\nabla/i)} P_N u \Vert_{L_t^2([0,N^{-2}], L_{x,y}^\infty)} \lesssim N^{-1/3} \Vert P_N u \Vert_{L^2(\mathbb{T}^2)}
\end{equation}
This yields the worse bilinear estimate
\begin{equation*}
\label{eq:bilinearEstimateHolder}
\begin{split}
&\Vert e^{it \varphi_{ZK}(\nabla/i)} P_N u_1 e^{it \varphi_{ZK}(\nabla/i)} P_K u_2 \Vert_{L_t^2([0,N^{-2}],L^2(\mathbb{T}^2))} \\
&\lesssim \Vert e^{it \varphi_{ZK}(\nabla/i)} P_N u \Vert_{L_t^2([0,N^{-2}], L_{x,y}^\infty)} \Vert e^{ it \varphi_{ZK}(\nabla/i)} P_K u_2 \Vert_{L_t^\infty L^2(\mathbb{T}^2)} \\
&\lesssim N^{-1/3} \Vert P_N u_1(0) \Vert_{L^2(\mathbb{T}^2)} \Vert P_K u_2(0) \Vert_{L^2(\mathbb{T}^2)}
\end{split}
\end{equation*}

\section{Notation and Function Spaces}
\label{section:FunctionSpaces}
We use shorttime $U^p$-/$V^p$-spaces which have previously been deployed for functions defined on the real line (see e.g. \cite{ChristHolmerTataru2012, KochTataru2012}). Here, we adapt to periodic solutions. For a detailed exposition on $U^p$-/$V^p$-spaces we refer to \cite{HadacHerrKoch2009}, see also \cite{HadacHerrKoch2009Addendum}. Below we collect the most important information to keep the exposition self-contained. The $V^p(I)$-spaces contain functions of bounded $p$-variation, where $p \in [1,\infty)$ and $U^p(I)$ is an atomic space which are the predual spaces of the $V^p(I)$-spaces. Here, $I=[a,b)$, where $-\infty \leq a < b \leq \infty$ and the functions under consideration will take values in $L^2(\T^n)$ (although the function space properties will remain valid for an arbitrary Hilbert space). We let $\mathcal{Z}(I)$ denote the set of all possible partitions of $I$, that are sequences $a=t_0 < t_1 < \ldots < t_K = b$.
\begin{definition}
Let $\{ t_k \}_{k=0}^K \in \mathcal{Z}(I)$ and $\{ \phi_k \}_{k=0}^{K-1} \subseteq L_x^2$ with $\sum_{k=1}^K \Vert \phi_{k-1} \Vert_{L_x^2}^p = 1$. Then, the function
\begin{equation}
a(t) = \sum_{k=1}^K \phi_{k-1} \chi_{[t_{k-1},t_k)}(t)
\end{equation}
is said to be a $U^p(I)$-atom. Further,
\begin{equation}
U^p(I)= \{ f: I \rightarrow L_x^2(\T^n) \, | \, \Vert f \Vert_{U^p(I)} < \infty \}
\end{equation}
where
\begin{equation}
\Vert f \Vert_{U^p(I)} = \inf \{ \Vert \lambda_k \Vert_{\ell_k^1} \, | \, f(t) = \sum_{k=0}^\infty \lambda_k a_k(t), \; a_k - U^p-\text{atom} \}
\end{equation}
\end{definition}
From the atomic representation we find elements $u(t) \in U^p(I)$ to be continuous from the right, having left-limits everywhere and admitting only countably many discontinuities (cf. \cite[Proposition~2.2,~p.~921]{HadacHerrKoch2009}). Properties of the spaces with bounded $p$-variation were already discussed in \cite{Wiener1979}.
\begin{definition}
We set
\begin{equation*}
V^p(I) = \{ v: I \rightarrow L_x^2 \, | \, \Vert v \Vert_{V^p(I)} < \infty \},
\end{equation*}
where
\begin{equation*}
\Vert v \Vert_{V^p(I)} = \sup_{\{t_k \}_{k=0}^{K-1} \in \mathcal{Z}(I)} \left( \sum_{k=1}^K \Vert v(t_k) - v(t_{k-1}) \Vert_{L_x^2}^p \right)^{1/p} < \infty
\end{equation*}
\end{definition}
We recall that one-sided limits exist for $V^p$-functions and again $V^p$-functions can only have countably many discontinuities (cf. \cite[Proposition~2.4,~p.~922]{HadacHerrKoch2009}).
In the following we will confine ourselves to considering the subspaces $V^p_{-,rc} \subseteq V^p$ of right-continuous functions vanishing at $-\infty$. For the sake of brevity we will write $V^p$ for $V^p_{-,rc}$.
\begin{definition}
We define the following subspaces of $V^2$, respectively $U^2$:
\begin{equation*}
\begin{split}
V_0^2(I) &= \{ v \in V^2(I) \, | \, v(a) = 0 \} \\
U_0^2(I) &= \{ u \in U^2(I) \, | \, u(b) = 0 \}
\end{split}
\end{equation*}
\end{definition}
These function spaces behave well with sharp cutoff functions contrary to $X^{s,b}$-spaces, where one has to use smooth cutoff functions. We have the following estimates for sharp cut-offs (cf. \cite[Equation~(2.2),~p.~55]{ChristHolmerTataru2012}):
\begin{equation*}
\begin{split}
\Vert u \Vert_{U^p(I)} &= \Vert \chi_I u \Vert_{U^p([-\infty,\infty))} \\
\Vert v \Vert_{V^p(I)} &\leq \Vert \chi_I u \Vert_{V^p([-\infty,\infty))} \leq 2 \Vert u \Vert_{V^p(I)}
\end{split}
\end{equation*}
We further record the following embedding properties:
\begin{lemma}
\label{lem:UpVpFunctionSpaceProperties}
Let $I = [a,b)$.
\begin{enumerate}
\item[1.] If $1 \leq p \leq q < \infty$, then $\Vert u \Vert_{U^q} \leq \Vert u \Vert_{U^p}$ and $\Vert u\Vert_{V^q} \leq \Vert u \Vert_{V^p}$.
\item[2.] If $1 \leq p < \infty$, then $\Vert u \Vert_{V^p} \lesssim \Vert u \Vert_{U^p}$.
\item[3.] If $1 \leq p < q < \infty$, $u(a) = 0$ and $u \in V^p$ is right-continuous, then $\Vert u \Vert_{U^q} \lesssim \Vert u \Vert_{V^p}$.
\item[4.] Let $1 \leq p < q <\infty$, $E$ be a Banach space and $T$ be a linear operator with
\begin{equation*}
\Vert Tu \Vert_E \leq C_q \Vert u \Vert_{U^q}, \; \Vert Tu \Vert_E \leq C_p \Vert u \Vert_{U^p}, \text{with} \; 0 < C_p \leq C_q.
\end{equation*}
Then,
\begin{equation*}
\Vert T u \Vert_E \lesssim \log \langle \frac{C_q}{C_p} \rangle \Vert u \Vert_{V^p}
\end{equation*}
\end{enumerate}
\end{lemma}
\begin{proof}
The first part follows from the embedding properties of the $\ell^p$-norms and the second part from considering $U^p$-atoms. For the third claim see \cite[Corollary~2.6,~p.~923]{HadacHerrKoch2009}.
\end{proof}

\begin{definition}
We define
\begin{equation*}
DU^2(I) = \{ \partial_t u \, | \, u \in U^2(I) \}
\end{equation*}
with the derivative taken in the sense of generalized functions.
\end{definition}
We observe that for any $f \in DU^2(I)$, the function $u \in U^2(I)$ satisfying $\partial_t u = f$ is unique up to constants. Fixing the right limit to be zero, we can set
\begin{equation*}
\Vert f \Vert_{DU^2(I)} = \Vert u \Vert_{U^2(I)}, \quad f = \partial_t u, \quad u \in U_0^2
\end{equation*}
which makes $DU^2(I)$ a Banach space. We have the following embedding property (cf. \cite[p.~56]{ChristHolmerTataru2012}):
\begin{lemma}
Let $I = [a,b)$. Then,
\begin{equation*}
L^1(I) \hookrightarrow DU^2(I)
\end{equation*}
\end{lemma}
We have the following lemma on $DU-V$-duality:
\begin{lemma}{\cite[Proposition~2.10,~p.~925]{HadacHerrKoch2009}}
We have $(DU^2(I))^* = V_0^2(I)$ with respect to a duality relation which for $f \in L^1(I) \subseteq DU^2(I)$ is given by
\begin{equation*}
\langle f, v \rangle = \int_a^b \langle f(t), v(t) \rangle_{L_x^2} dt = \int_a^b \int f \overline{v} dx dt
\end{equation*}
Moreover,
\begin{equation*}
\Vert f \Vert_{DU^2(I)} = \sup_{\Vert v \Vert_{V_0^2} = 1} \left| \int_a^b \int f \overline{v} dx dt \right|
\end{equation*}
\end{lemma}
For $f \in DU^2(I)$ one can still consider a related mapping, but this requires more careful considerations (cf. \cite[Theorem~2.8,~p.~924]{HadacHerrKoch2009}).\\
Next, we consider a linear dispersive PDE on $\mathbb{T}^n$:
\begin{equation*}
i \partial_t u + \varphi(\nabla/i) u = 0, \quad u: \mathbb{R} \times \mathbb{T}^n \rightarrow \mathbb{C}
\end{equation*}
Adapting $U^p$-/$V^p$-spaces to the linear propagator $e^{i t \varphi(\nabla/i)}$ yields the following function spaces:
\begin{equation*}
\begin{split}
\Vert u \Vert_{U^p_{\varphi}(I;H)} &= \Vert e^{-it \varphi(\nabla/i)} u \Vert_{U^p(I;H)} \\
\Vert v \Vert_{V^p_{\varphi}(I;H)} &= \Vert e^{-it \varphi(\nabla/i)} v \Vert_{V^p(I;H)} \\
\Vert u \Vert_{DU^2_{\varphi}(I;H)} &= \Vert e^{-it \varphi(\nabla/i)} u \Vert_{DU^2(I;H)}
\end{split}
\end{equation*}
$U^p_\varphi$-atoms are piecewise free solutions.\\
Suppose that $\varphi$ satisfies \eqref{eq:alphaGroupVelocity} for some $\alpha > 1$. That means that waves with frequency localization $N$ travel with a speed of $N^{\alpha-1}$. Consequently, in a time interval of size $N^{-(\alpha-1)}$ there should be no difference observing the waves on $\mathbb{T}^n$ or $\mathbb{R}^n$. Correspondingly, we define the shorttime spaces for $\varphi$ with a time localization of order $N^{-(\alpha-1)}$ for frequencies $N$.\\
The frequency projector is defined as
\begin{equation*}
(P_N f) \widehat (\xi) = 
\begin{cases}
1_{[N,2N)}(|\xi|) \hat{f}(\xi), \quad N \in 2^{\mathbb{N}} \\
1_{[0,2)}(|\xi|) \hat{f}(\xi), \quad N = 1
\end{cases}
\end{equation*}
We define the shorttime $U^2$-space into which we will place the solution as
\begin{equation}
\Vert u \Vert^2_{F^s} = \sum_{N \geq 1} N^{2s} \sup_{|I| = N^{-(\alpha-1)}} \Vert \chi_I P_N u \Vert_{U^2_{\varphi}(I;L^2)}^2
\end{equation}
The space into which we will place the nonlinearity is given as
\begin{equation}
\Vert f \Vert^2_{N^s} = \sum_{N \geq 1} N^{2s} \sup_{|I| = N^{-(\alpha-1)}} \Vert \chi_I P_N u \Vert_{DU^2_{\varphi}(I;L^2)}^2
\end{equation}
The frequency dependent time localization erases the dependence on the initial data away from the origin. Instead of a common energy space $C([0,T],H^s)$ this forces us to consider the following space: 
\begin{equation}
\Vert u \Vert^2_{E^s(T)} = \sum_{N \geq 1} N^{2s} \sup_{t \in [0,T]} \Vert P_N u(t) \Vert^2_{L^2}
\end{equation}
This space deviates from the usual energy space logarithmically.\\
Next, consider the nonlinear equation
\begin{equation}
\label{eq:nonlinearDispersivePDE}
i \partial_t u + \varphi(\nabla/i) u = F(u), \quad u: \mathbb{R} \times \mathbb{T}^n \rightarrow \mathbb{C}.
\end{equation}
with $\varphi$ satisfying \eqref{eq:alphaGroupVelocity} for some $\alpha > 1$. The propagation of $u$ in the $F^s$-spaces is described by the following lemma:
\begin{lemma}
Let $u$ be a solution to \eqref{eq:nonlinearDispersivePDE}. Then, we find the following estimate to hold:
\begin{equation}
\Vert u \Vert_{F^s_{\varphi}} \lesssim \Vert u \Vert_{E^s(T)} + \Vert F(u) \Vert_{N_\varphi^s(T)}
\end{equation}
\end{lemma}
\begin{proof}
Consider for $N \geq 1$ some $I = [t_0,t_1] \subseteq [0,1]$, $|I|=N^{-(\alpha-1)}$. Afer projecting \eqref{eq:nonlinearDispersivePDE} to frequencies of size $N$ we find
\begin{equation*}
P_N u(t) = e^{i(t-t_0) \varphi(\nabla/i)} P_N u(t_0) - i \int_{t_0}^t e^{i(t-s) \varphi(\nabla/i)} P_N F(u)(s) ds
\end{equation*}
on $I$. The claim follows from the above display by the definition of the function spaces and unitarity of the propagator.
\end{proof}
Since $U^2_\varphi$-atoms are piecewise free solutions, we can infer the following estimates. This conclusion is commonly referred to as transfer principle.
\begin{proposition}
Let $N_1 \gg N_2 \sim N_3$ and $I$ and $J$ be intervals with $|I|=N_1^{-1}$, $|J|=N_2^{-1}$. Then, we find the following estimates to hold:
\begin{equation*}
\begin{split}
\Vert P_{N_1} u_1 P_{N_2} u_2 \Vert_{L^2(I,L^2(\mathbb{T}))} &\lesssim N_1^{-1/2} \Vert P_{N_1} u_1 \Vert_{U^2_{BO}(I)} \Vert P_{N_2} u_2 \Vert_{U^2_{BO}(I)} \\
\Vert P_{N_1} u_1 P_{N_2} u_2 \Vert_{L^2(I,L^2(\mathbb{T}))} &\lesssim N_1^{-1/2} \log \langle \frac{N_1}{N} \rangle \Vert P_{N_1} u_1 \Vert_{V^2_{BO}(I)} \Vert P_{N_2} u_2 \Vert_{V^2_{BO}(I)} \\
\Vert S_{\gtrsim N_2} (P_{N_2} u_2 P_{N_3} u_3) \Vert_{L^2(J,L^2(\mathbb{T}))} &\lesssim N_2^{-1/2} \Vert P_{N_2} u_2 \Vert_{U^2_{BO}(J)} \Vert P_{N_3} u_3 \Vert_{U^2_{BO}(J)}
\end{split}
\end{equation*}
\end{proposition}
\begin{proof}
The $U^2_{BO}$-estimates are a consequence of atomic representations like described above. For the $V^2_{BO}$-estimate use Lemma \ref{lem:UpVpFunctionSpaceProperties}, Property 4. For details on the interpolation see e.g. \cite{HadacHerrKoch2009}. 
\end{proof}
The analogous bilinear estimates for \eqref{eq:ZakharovKuznetsovEquation} are omitted.

\section{Applications of the shorttime bilinear Strichartz estimates to shorttime nonlinear estimates}
\label{section:ShorttimeNonlinearEstimates}
In this section the nonlinearity of certain dispersive PDE is propagated in shorttime norms. We focus on quadratic derivative nonlinearities, that is we will consider
\begin{equation*}
i \partial_t u + \varphi(\nabla/i) u = \partial_x (u^2), \quad u: \mathbb{R} \times \mathbb{T}^n \rightarrow \mathbb{R}
\end{equation*}
in order to keep the exposition simple. In Section \ref{section:HigherOrderGeneralizations} we generalize to higher order nonlinearities $\partial_x (u^k)$, $k \geq 3$.\\
Although one could formulate the proof of the estimates
\begin{equation*}
\begin{split}
\Vert \partial_x (uv) \Vert_{N^s_\varphi} &\lesssim \Vert u \Vert_{F^s_\varphi} \Vert v \Vert_{F^s_{\varphi}} \\
\Vert \partial_x (uv) \Vert_{N^0_\varphi} &\lesssim \Vert u \Vert_{F^0_\varphi} \Vert v \Vert_{F^s_\varphi}
\end{split}
\end{equation*}
in an abstract way, we prefer to prove the bilinear estimates in the two cases of the Benjamin-Ono equation and the Zakharov-Kuznetsov equation. We will see that the arguments parallel each other, although the concretely deployed Strichartz estimates differ. Namely, after performing a Littlewood-Paley decomposition $u = \sum_{N \geq 1} P_N u$ we have to consider the following three interactions in both cases:
\begin{enumerate}
\item[$\bullet$] $High \times Low \rightarrow High$-interaction: Suppose that $N_2 \ll N_1 \sim N$. We have to prove the estimate
\begin{equation*}
\Vert P_N \partial_x (P_{N_1} u P_{N_2} u) \Vert_{N_\varphi} \lesssim N_2^{s-\varepsilon} \Vert P_{N_1} u \Vert_{F_\varphi} \Vert P_{N_2} v \Vert_{F_\varphi}
\end{equation*}
where the $\varepsilon$ is required to carry out the square sum in $K$.
\item[$\bullet$] $High \times High \rightarrow High$-interaction: Suppose that $N_1 \sim N_2 \sim N$. We have to prove the estimate
\begin{equation*}
\Vert P_{N} \partial_x (P_{N_1} u P_{N_2} v) \Vert_{N_\varphi} \lesssim N^s \Vert P_{N_1} u \Vert_{F_\varphi} \Vert P_{N_2} v \Vert_{F_\varphi}
\end{equation*}
\item[$\bullet$] $High \times High \rightarrow Low$-interaction: Suppose that $N_1 \sim N_2 \gg N$. We have to show
\begin{equation*}
\Vert P_N \partial_x (P_{N_1} u P_{N_2} v) \Vert_{N_\varphi} \lesssim N^{-s-\varepsilon} N_1^{2s} \Vert P_{N_1} u \Vert_{F_\varphi} \Vert P_{N_2} u \Vert_{F_\varphi}
\end{equation*}
\end{enumerate}
\subsection{Benjamin-Ono equation}
In case of the Benjamin-Ono equation the derived estimates take the following form.
\begin{proposition}
\label{prop:nonlinearEstimatesBO}
Let $T \in (0,1]$ and $s>0$. Then we find the following estimates to hold:
\begin{eqnarray}
\label{eq:nonlinearEstimateBOI}
\Vert \partial_x (uv) \Vert_{N^s_{BO}(T)} &\lesssim \Vert u \Vert_{F^s_{BO}(T)} \Vert v \Vert_{F^s_{BO}(T)} \\
\label{eq:nonlinearEstimateBOII}
\Vert \partial_x (uv) \Vert_{N^0_{BO}(T)} &\lesssim \Vert u \Vert_{F^0_{BO}(T)} \Vert v \Vert_{F^s_{BO}(T)}
\end{eqnarray}
\end{proposition}
\begin{proof}
In case of $High \times Low \rightarrow High$-interaction we use the embedding $L^1(I) \hookrightarrow DU^2_{BO}(I)$, H\"older in time (recall that $|I|=N^{-1})$ and the bilinear Strichartz estimate to derive
\begin{equation*}
\begin{split}
\Vert P_{N} \partial_x (P_{N_1} u P_{N_2} v) \Vert_{N^0_{BO}} &\lesssim N \Vert P_{N_1} u P_{N_2} v \Vert_{L^1(I;L_x^2)} \\
&\lesssim N^{1/2} \Vert P_{N_1} u P_{N_2} v \Vert_{L^2(I;L_x^2)} \lesssim \Vert P_{N_1} u \Vert_{U^2_{BO}(I)} \Vert P_{N_2} v \Vert_{U^2_{BO}(I)}
\end{split}
\end{equation*}
For $High \times High \rightarrow High$-interaction we use duality to write
\begin{equation*}
\Vert P_N \partial_x (P_{N_1} u P_{N_2} v) \Vert_{N_{BO}} = \sup_{\Vert w \Vert_{V_0^2} = 1} N \left| \int \int P_{N^\prime} w P_{N_1} u P_{N_2} v dx dt \right|
\end{equation*}
Since two factors must be frequency separated of order $N$, we can use a bilinear Strichartz estimate on two factors (say $w$ and $u$) and use the energy estimate on the remaining factor to find
\begin{equation*}
\begin{split}
N \left| \int_I \int P_{N^\prime} w P_{N_1} u P_{N_2} v dx dt \right| &\lesssim N \Vert P_{N^\prime} w P_{N_1} u \Vert_{L^2(I;L^2)} \Vert P_{N_2} v \Vert_{L^2(I;L^2)} \\
&\lesssim \log (N) \Vert w \Vert_{V^2_{BO}(I)} \Vert P_{N_1} u \Vert_{V^2_{BO}} \Vert P_{N_2} v \Vert_{U^2_{BO}}
\end{split}
\end{equation*}
Finally, for $High \times High \rightarrow Low$-interaction we have to partition the interval $I$, $|I| = N^{-1}$ into $N_1^{-1}$ intervals which accounts for a factor of $N_1/N$. Using duality and the bilinear Strichartz estimate we find
\begin{equation*}
\begin{split}
\frac{N_1}{N} \cdot N \int_{I^\prime} \int P_N w P_{N_1} u P_{N_2} v dx dt &\lesssim N_1 \Vert P_N w P_{N_1} u \Vert_{L^2(I^\prime,L_x^2)} \Vert P_{N_2} v \Vert_{L^2(I^\prime,L_x^2)} \\
&\lesssim \log \langle \frac{N_1}{N} \rangle \Vert P_N w \Vert_{V^2_{BO}(I^\prime)} \Vert P_{N_1} u \Vert_{V^2_{BO}(I^\prime)} \Vert P_{N_2} v \Vert_{U^2_{BO}}
\end{split}
\end{equation*}
which yields the claim.
\end{proof}
\subsection{Zakharov-Kuznetsov equation}
\begin{proposition}
\label{prop:nonlinearEstimatesZK}
Let $T \in (0,1]$ and $s>0$. Then we find the following estimates to hold:
\begin{eqnarray}
\label{eq:nonlinearEstimateZKI}
\Vert \partial_x (uv) \Vert_{N^s_{ZK}(T)} \lesssim \Vert u \Vert_{F^s_{ZK}(T)} \Vert v \Vert_{F^s_{ZK}(T)} \\
\label{eq:nonlinearEstimateZKII}
\Vert \partial_x (uv) \Vert_{N^0_{ZK}(T)} \lesssim \Vert u \Vert_{F^0_{ZK}(T)} \Vert v \Vert_{F^s_{ZK}(T)}
\end{eqnarray}
\end{proposition}
\begin{proof}
For $High \times Low \rightarrow High$-interaction we compute like above for $N_2 \ll N \sim N_1$
\begin{equation*}
\begin{split}
\Vert \partial_x P_N (P_{N_1} u P_{N_2} u) \Vert_{N^0_{ZK}} &\lesssim N \Vert P_N (P_{N_1} u P_{N_2} v) \Vert_{DU^2_{ZK}} \lesssim \Vert P_{N_1} u P_{N_2} v \Vert_{L^2([0,N^{-2}],L^2)} \\
&\lesssim \frac{K^{1/2}}{N} \Vert P_{N_1} u \Vert_{U^2_{ZK}} \Vert P_{N_2} v \Vert_{U^2_{ZK}}
\end{split}
\end{equation*}
For $High \times High \rightarrow High$-interaction we find like above using duality and by the remark following the proof of Proposition \ref{prop:shorttimeBilinearStrichartzEstimate} for $N \sim N_1 \sim N_2$
\begin{equation*}
\Vert P_N (P_{N_1} u P_{N_2} u) \Vert_{N_{ZK}} \lesssim \log \langle N \rangle N^{-1/2} \Vert P_{N_1} u \Vert_{U^2_{ZK}} \Vert P_{N_2} u \Vert_{U^2_{ZK}}
\end{equation*}
For $High \times High \rightarrow Low$-interaction we again add localization in time and use bilinear estimates for $V^2$-functions to find like in the proof of Proposition \ref{prop:nonlinearEstimatesBO}
\begin{equation*}
\Vert P_N \partial_x(P_{N_1} u P_{N_2} v) \Vert_{N_{ZK}} \lesssim N \log \langle \frac{N_1}{N} \rangle \left( \frac{N_1}{N} \right)^2 \frac{N^{1/2}}{N_1} (N_1)^{-1} \Vert P_{N_1} u \Vert_{U^2_{ZK}} \Vert P_{N_2} v \Vert_{U^2_{ZK}}
\end{equation*}
Since one has to increase the frequency dependent localization in time giving a factor $(N_1/N)^2$ this estimate is the worst one and gives the threshold $s>0$.
The proof is complete.
\end{proof}
\section{Energy estimates}
\label{section:EnergyEstimates}
Purpose of this section is to propagate the energy norm. For solutions to the original equation this estimate reads as
\begin{equation}
\label{eq:EnergyNormPropagationSolution}
\Vert u \Vert^2_{E^s(T)} \lesssim \Vert u_0 \Vert_{H^s}^2 + \Vert u \Vert^3_{F^s(T)}
\end{equation}
For solutions to the difference equation, that is for $v = u_1-u_2$, where $u_1,u_2$ are solutions to the original equation we prove two estimates in addition to \eqref{eq:EnergyNormPropagationSolution}. The first one leads to Lipschitz dependence of the data-to-solution mapping in $L^2$:
\begin{equation}
\label{eq:L2PropagationDifferences}
\Vert v \Vert^2_{E^0(T)} \lesssim \Vert v(0) \Vert^2_{L^2} + \Vert v \Vert^2_{F^0(T)} ( \Vert u_1 \Vert_{F^s(T)} + \Vert u_2 \Vert_{F^s(T)})
\end{equation}
The second one will lead to non-uniform continuous dependence of the data-to-solution mapping in $H^s$:
\begin{equation}
\label{eq:EnergyNormPropagationDifferences}
\Vert v \Vert^2_{E^s(T)} \lesssim \Vert v \Vert^2_{H^s} + \Vert v \Vert^3_{F^s(T)} + \Vert v \Vert^2_{F^s(T)} \Vert u_2 \Vert_{F^s(T)} + \Vert v \Vert_{F^0(T)} \Vert v \Vert_{F^s(T)} \Vert u_2 \Vert_{F^{2s}(T)}
\end{equation}
In case of the Benjamin-Ono equation we prove the following estimates:
\begin{proposition}
\label{prop:EnergyTransferBenjaminOnoEquation}
Let $T \in (0,1]$ and $s > 1$.
\begin{enumerate}
\item[(a)] For a smooth solution $u$ to \eqref{eq:BenjaminOnoEquation} we find \eqref{eq:EnergyNormPropagationSolution} to hold.
\item[(b)] Let $u_1$, $u_2$ be smooth solutions to \eqref{eq:BenjaminOnoEquation} and $v=u_1-u_2$ be the difference of the two solutions. Then we find \eqref{eq:L2PropagationDifferences} and \eqref{eq:EnergyNormPropagationDifferences} to hold.
\end{enumerate}
\end{proposition}
The building blocks to prove Proposition \ref{prop:EnergyTransferBenjaminOnoEquation} will be the estimates proven in the following lemma.
\begin{lemma}
Let $T \in (0,1]$ and $N_2 \leq N_1 \sim N$. Then we find the following estimate to hold:
\begin{equation}
\label{eq:FrequencyLocalizedEnergyTransferBenjaminOno}
\int_0^T ds \int dx P_N u_1 P_{N_1} u_2 P_{N_2} u_3 \lesssim \Vert P_N u_1 \Vert_{F^0} \Vert P_{N_1} u_2 \Vert_{F^0} \Vert P_{N_2} u_3 \Vert_{F^0}
\end{equation}
\end{lemma}
\begin{proof}
The key ingredient is again the shorttime bilinear Strichartz estimate. First, consider the case $N_2 \ll N$. After breaking $[0,T]$ into $\lesssim NT$ intervals $I$ of size $N^{-1}$ we have to estimate
\begin{equation*}
\begin{split}
\int_I ds \int dx P_N u_1 P_{N_1} u_2 P_{N_2} u_3 &\leq \Vert P_N u_1 P_{N_1} u_2 \Vert_{L^2(I;L^2)} \Vert P_{N_2} u_3 \Vert_{L^2(I;L^2)} \\
&\lesssim N^{-1} \Vert P_N u_1 \Vert_{U^2_{BO}(I;L^2)} \Vert P_{N_2} u_2 \Vert_{U^2_{BO}(I;L^2)} \Vert P_{N} u_3 \Vert_{U^2_{BO}(I;L^2)}
\end{split}
\end{equation*}
Since splitting the time interval accounts for a factor of at most $N$, the proof is complete. In case $N_2 \sim N$ we can still use a shorttime bilinear Strichartz estimate following Remark \ref{remark:EuclideanTransversalityExtension}.
\end{proof}
We show Proposition \ref{prop:EnergyTransferBenjaminOnoEquation}:
\begin{proof}[Proof of Proposition \ref{prop:EnergyTransferBenjaminOnoEquation}]
First, we show \eqref{eq:EnergyNormPropagationSolution}. One has to analyze $\sup_{t \in [0,T]} \Vert P_N u(t) \Vert^2_{L^2}$ in order to conclude the estimate of the $E^s$-norm after carrying out the sum over $N$ with weight $N^{2s}$.\\
The fundamental theorem of calculus yields
\begin{equation*}
\Vert P_N u(t) \Vert^2_{L^2} = \Vert P_N u(0) \Vert^2_{L^2} + 2 \int_0^t ds \int dx P_N u P_N (\partial_x (u u ))
\end{equation*}
First, consider $High \times Low \rightarrow High$-interaction. That means we estimate
\begin{equation*}
\int_0^t ds \int dx P_N u P_N [ \partial_x(u P_{N_2} u)],
\end{equation*}
where $N_2 \ll N$. In case the derivative hits the high frequency factor we integrate by parts in order to derive a favourable expression: Write
\begin{equation*}
\begin{split}
\int_0^t ds \int dx P_N u P_N (\partial_x u P_{N_2} u) &= \int_0^t ds \int dx P_N u P_N(\partial_x u) P_{N_2} u dx \\
 &+ \int_0^t ds \int dx P_N u [ P_N (\partial_x u P_{N_2} u) - P_N (\partial_x u) P_{N_2} u] dx
 \end{split}
\end{equation*}
After integration by parts the first expression is dominated by $|\int ds \int dx P_N u P_N u P_{N_2}(\partial_x u)|$.\\
For the second one consider the squared bracket in momentum space:
\begin{equation*}
\begin{split}
&\chi_N(\xi_1+\xi_2) (-i \xi_1) \hat{u}(\xi_1) \chi_{N_2}(\xi_2) \hat{u}(\xi_2) - \chi_N(\xi_1) (-i\xi_1) \hat{u}(\xi_1) \chi_{N_2}(\xi_2) \hat{u}(\xi_2) \\
&= \{ \chi_N(\xi_1+\xi_2) - \chi_N(\xi_1) \} (-i\xi_1) \hat{u}(\xi_1) \chi_{N_2}(\xi_2) \hat{u}(\xi_2)
\end{split}
\end{equation*}
The mean value theorem shows that the expression in the squared brackets is of order at most $N_2/N$. Since it is smooth and admits an off-diagonal continuation, we can expand it into a rapidly converging Fourier series and changing back to position space we find
\begin{equation*}
\frac{N_2}{N} \int_0^t ds \int dx P_N u \tilde{P}_N (\partial_x u P_{N_2} u)
\end{equation*}
which implies that effectively we are again dealing with an expression of the form
\begin{equation*}
\sum_{N_1 \sim N} \left| \int ds \int dx P_N u P_{N_1} u P_{N_2}(\partial_x u) \right|
\end{equation*}
By virtue of \eqref{eq:FrequencyLocalizedEnergyTransferBenjaminOno} and the Cauchy-Schwarz inequality one finds
\begin{equation*}
\begin{split}
&\sum_{N \geq 1} N^{2s} \sum_{N_1 \sim N} \sum_{1 \leq N_2 \leq N} \left| \int_0^t ds \int dx P_N u P_{N_1} u P_{N_2}(\partial_x u) \right| \\
&\lesssim \sum_{N \geq 1} N^{2s} \sum_{N_1 \sim N} \sum_{1 \leq N_2 \leq N} N_2 \Vert P_N u \Vert_{F^0} \Vert P_{N_1} u \Vert_{F^0} \Vert P_{N_2} u \Vert_{F^0} \lesssim \Vert u \Vert^3_{F^s(T)}
\end{split}
\end{equation*}
In the above estimate we did not distinguish between $High \times Low \rightarrow High$-interaction or $High \times High \rightarrow High$-interaction. When dealing with $High \times High \rightarrow High$-interaction there is no point in integrating by parts.\\
In case of $High \times High \rightarrow Low$-interaction we again do not integrate by parts, but simply use \eqref{eq:FrequencyLocalizedEnergyTransferBenjaminOno} to find
\begin{equation*}
\begin{split}
&\sum_{N \geq 1} N^{2s} \sum_{N_1 \sim N_2 \gg N} \left| \int ds \int dx P_N u P_N \partial_x (P_{N_1} u P_{N_2} u) \right| \\
&\lesssim \sum_{N \geq 1} N^{2s+1} \sum_{N_1 \sim N_2 \gg N} \Vert P_N u \Vert_{F^0(T)} \Vert P_{N_1} u \Vert_{F^0(T)} \Vert P_{N_2} u \Vert_{F^0(T)} \lesssim \Vert u \Vert^3_{F^s(T)}
\end{split}
\end{equation*}
provided that $s > 1$. The proof of estimate \eqref{eq:EnergyNormPropagationSolution} is complete.\\
Next, we turn to estimate \eqref{eq:L2PropagationDifferences}: Due to the reduced symmetry one can not always integrate by parts like above. Again, we invoke the fundamental theorem of calculus to write
\begin{equation*}
\Vert P_N v(t) \Vert^2_{L^2} = \Vert P_N v(0) \Vert^2_{L^2} + 2 \int_0^t ds \int dx P_N v (P_N(\partial_x (v (u_1 + u_2)))) 
\end{equation*}
First, consider $High \times Low \rightarrow High$-interaction. In case the high frequency is on the difference solution, that means we are considering the expression 
\begin{equation*}
\int_0^t ds \int dx P_N v P_N(\partial_x(v(P_K (u_1 + u_2)))
\end{equation*}
we can argue like above with integration by parts and the commutator estimate to conclude
\begin{equation*}
\begin{split}
&\sum_{N \geq 1} \sum_{1 \leq N_2 \ll N} \sum_{N_1 \sim N} \Vert P_N v \Vert_{F^0} \Vert P_{N_1} v \Vert_{F^0} N_2 \Vert P_{N_2} u \Vert_{F^0} \\
&\lesssim \sum_{N \geq 1 } \sum_{N_1 \sim N} \Vert P_N v \Vert_{F^0(T)} \Vert P_{N_1} v \Vert_{F^0(T)} \Vert u \Vert_{F^s(T)} \lesssim \Vert v \Vert^2_{F^0} \Vert u \Vert_{F^s}
\end{split}
\end{equation*}
However, when considering the expression $\int ds \int dx P_N v P_N(\partial_x (P_{N_2} v \cdot u))$ we can not integrate by parts, but still
\begin{equation*}
\begin{split}
&\sum_{N \geq 1} \sum_{1 \leq N_2 \ll N} \sum_{N_1 \sim N} N \Vert P_N v \Vert_{F^0} \Vert P_{N_1} u \Vert_{F^0} \Vert P_{N_2} v \Vert_{F^0} \\
&\lesssim \Vert v \Vert^2_{F^0(T)} \Vert u \Vert_{F^s(T)}
\end{split}
\end{equation*}
provided that $s >1$.
In case of $High \times High \rightarrow High$- and $High \times High \rightarrow Low$-interaction the argument from the proof of \eqref{eq:EnergyNormPropagationSolution} applies without modification and yields the desired estimate. This completes the proof of \eqref{eq:L2PropagationDifferences}.\\
We turn to the proof of \eqref{eq:EnergyNormPropagationDifferences}. For this purpose rewrite the equation satisfied by $v = u_1 - u_2$ as
\begin{equation*}
\partial_t v + \mathcal{H} \partial_{xx} v = \partial_x (v^2) + \partial_x (v u_2)
\end{equation*}
Using the same strategy like above we have to focus on $High \times Low \rightarrow High$-interaction in the expression $\partial_x (v u_2)$. More precisely, we have to carry out the estimate
\begin{equation*}
\begin{split}
&\sum_{N \geq 1} N^{2s} \sum_{1 \leq N_2 \ll N} \int_0^t ds \int dx P_N v P_N(\partial_x u_2 P_{N_2} v) \\
&\lesssim \sum_{N \geq 1 } \sum_{1 \leq N_2 \ll N} \sum_{N \sim N_1} N^{2s+1} \Vert P_N v \Vert_{F^0} \Vert P_{N_1} u_2 \Vert_{F^0} \Vert P_{N_2} v \Vert_{F^0} \\
&\lesssim \Vert v \Vert_{F^0} \Vert v \Vert_{F^s} \Vert u_2 \Vert_{F^{2s}}
\end{split}
\end{equation*}
The remaining cases can be treated like above. The proof is complete.
\end{proof}
We record the corresponding result for the Zakharov-Kuznetsov equation.
\begin{proposition}
Let $T \in (0,1]$ and $s > 3/2$.
\begin{enumerate}
\item[(a)] Suppose that $u$ is a smooth solution to \eqref{eq:ZakharovKuznetsovEquation}. Then, we find \eqref{eq:EnergyNormPropagationSolution} to hold.
\item[(b)] Suppose that $u_1$ and $u_2$ are smooth solutions to \eqref{eq:ZakharovKuznetsovEquation} and let $v = u_1 - u_2$. Then, we find \eqref{eq:L2PropagationDifferences} and \eqref{eq:EnergyNormPropagationDifferences} to hold.
\end{enumerate}
\end{proposition}
The proof follows the argument from above, but the worse shorttime bilinear Strichartz estimate accounts for worse regularity.
\section{A priori estimates and continuity of the data-to-solution mapping}
\label{section:Conclusion}
This section is devoted to the proof of a priori estimates and the existence and continuity of the data-to-solution mapping for quadratic derivative nonlinearities. For most of the time we confine ourselves to the analysis of the evolution of small initial data in the unit time interval. At the end of the section we generalize the argument to large initial data. The argument is standard by now (cf. \cite{IonescuKenigTataru2008}), but included for the sake of completeness. We only demonstrate it for the Benjamin-Ono equation.\\
We start with proving a priori estimates for periodic solutions to the Benjamin-Ono equation. The set of estimates we deploy to prove a priori estimates for periodic solutions to the Benjamin-Ono equation provided that $s > 1$ are
\begin{equation*}
\left\{\begin{array}{cl}
\Vert u \Vert_{F^s(T)} &\lesssim \Vert u \Vert_{E^s(T)} + \Vert \partial_x (u^2) \Vert_{N^s(T)} \\
\Vert u \partial_x u \Vert_{N^s(T)} &\lesssim \Vert u \Vert_{F^s(T)}^2 \\
\Vert u \Vert^2_{E^s(T)} &\lesssim \Vert u_0 \Vert_{H^s}^2 + \Vert u \Vert_{F^s(T)}^3 \end{array} \right.
\end{equation*}
Putting the estimates together one finds
\begin{equation*}
\Vert u \Vert^2_{F^s(T)} \lesssim \Vert u_0 \Vert_{H^s}^2 + \Vert u \Vert^3_{F^s(T)} + \Vert u \Vert_{F^s(T)}^4
\end{equation*}
Starting with small initial data $\Vert u_0 \Vert_{H^s} \leq \varepsilon$ and due to $\lim_{T \rightarrow 0} \Vert u \Vert_{E^s(T)} \leq 2 \Vert u_0 \Vert_{H^s}$ the claim follows from a standard bootstrap argument. The details are omitted.\\
Due to $F^s(T) \hookrightarrow L_T^\infty H^s$ this yields a priori estimates for the Benjamin-Ono equation provided that $s > 1$ and the same argument yields a priori estimates for the Zakharov-Kuznetsov equation for $s > 3/2$.\\
We turn to continuity of the data-to-solution mapping. For this purpose we will use a variant of the Bona-Smith approximation (cf. \cite{BonaSmith1975}). The key idea is to show that the solutions coming from frequency truncated initial data provide good approximations. The approximations are not uniform. Therefore, the argument does not yield uniform continuity of the data-to-solution mapping which one can not expect for equations with quadratic derivative nonlinearity (cf. \cite{KochTzvetkov2005, LinaresPantheeRobertTzvetkov2018}).\footnote{Notably, the data-to-solution mapping for the Benjamin-Ono equation posed on $\mathbb{T}$ for real initial data with vanishing mean is $C^\infty$ (cf. \cite{Molinet2008}), but this seems to be very special.} The approximation rate depends on the distribution of the Sobolev energy on the frequencies.\\
We start with Lipschitz continuity in the $L^2$-topology provided that the initial data are in $H^s$ for $s>1$. For this purpose consider the following estimates derived in the previous sections for $s > 1$:
\begin{equation*}
\left\{\begin{array}{cl}
\Vert u \Vert_{F^0(T)} &\lesssim \Vert u \Vert_{E^0(T)} + \Vert \partial_x (v (u_1+u_2)) \Vert_{N^0(T)} \\
\Vert \partial_x (v (u_1+u_2)) \Vert_{N^0(T)} &\lesssim \Vert v \Vert_{F^0(T)} ( \Vert u_1 \Vert_{F^s(T)} + \Vert u_2 \Vert_{F^s(T)}) \\
\Vert v \Vert^2_{E^0(T)} &\lesssim \Vert v(0) \Vert_{L^2}^2 + \Vert v \Vert^2_{F^0(T)} ( \Vert u_1 \Vert_{F^s(T)} + \Vert u_2 \Vert_{F^s(T)}) \end{array} \right.
\end{equation*}
Provided that $\Vert u_1 \Vert_{F^s(1)}$, $\Vert u_2(0) \Vert_{F^s(1)}$ are sufficiently small which is the case for small initial data $\Vert u_1(0) \Vert_{H^s}$, $\Vert u_2(0) \Vert_{H^s}$ according  to the above analysis we find from putting the estimates together
\begin{equation*}
\Vert v \Vert^2_{F^0(1)} \lesssim \Vert v(0) \Vert_{L^2}^2 + \Vert v \Vert^2_{F^0(1)} ( \Vert u_1 \Vert^2_{F^s(1)} + \Vert u_2 \Vert^2_{F^s(1)}) + \Vert v \Vert^2_{F^0(1)} (\Vert u_1 \Vert_{F^s(1)} + \Vert u_2 \Vert_{F^s(1)})
\end{equation*}
and for sufficiently small initial data this implies
\begin{equation*}
\Vert v \Vert_{F^0(1)} \lesssim_{\Vert u_i(0) \Vert_{H^s}} \Vert v(0) \Vert_{L^2}
\end{equation*}
Moreover, for differences of solutions to the Benjamin-Ono equation we find at $H^s$-regularity
\begin{equation*}
\left\{\begin{array}{cl}
\Vert v \Vert_{F^s(T)} &\lesssim \Vert v \Vert_{E^s(T)} + \Vert \partial_x (v(u_1+u_2)) \Vert_{N^s(T)} \\
\Vert \partial_x ( v(u_1+u_2)) \Vert_{N^s(T)} &\lesssim \Vert v \Vert_{F^s(T)} (\Vert u_1 \Vert_{F^s(T)} + \Vert u_2 \Vert_{F^s(T)}) \\
\Vert v \Vert^2_{E^s(T)} &\lesssim \Vert v(0) \Vert_{H^s}^2 + \Vert v \Vert^3_{F^s(T)} + \Vert v \Vert_{F^0(T)} \Vert v \Vert_{F^s(T)} \Vert u_2 \Vert_{F^{2s}(T)} \end{array} \right.
\end{equation*}
provided that $s>1$.\\
Let $S_T^\infty$ denote the mapping $H^\infty \rightarrow C([0,T],H^\infty)$ assigning smooth initial data to smooth solutions to the Benjamin-Ono equation.\\
Let $u_0$ be a smooth initial datum and consider $v = S_T^\infty(u_0) - S_T^\infty(P_{\leq N} u_0)$. Observe that
\begin{equation*}
\Vert v \Vert_{F^0(T)} \lesssim \Vert u_0 - P_{\leq N} u_0 \Vert_{L^2} \lesssim \Vert P_{> N} u_0 \Vert_{L^2} \lesssim N^{-s} \Vert P_{>N} u_0 \Vert_{H^s}
\end{equation*}
Moreover, $P_{> N} u_0$ is the initial datum to $v$. Consequently, $\Vert P_{> N} u_0 \Vert_{H^s} \lesssim \Vert v(0) \Vert_{H^s} \lesssim \Vert v \Vert_{F^s(T)}$. A variant of the proof of the a priori estimates for solutions yields the system of inequalities
\begin{equation*}
\left\{\begin{array}{cl}
\Vert u \Vert_{F^{2s}(T)} &\lesssim \Vert u \Vert_{E^{2s}(T)} + \Vert \partial_x (u^2) \Vert_{N^{2s}(T)} \\
\Vert u \partial_x u \Vert_{N^{2s}(T)} &\lesssim \Vert u \Vert_{F^{2s}(T)} \Vert u \Vert_{F^s(T)} \\
\Vert u \Vert^2_{E^{2s}(T)} &\lesssim \Vert u_0 \Vert_{H^{2s}}^2 + \Vert u \Vert_{F^{2s}(T)}^2 \Vert u \Vert_{F^s(T)}  \end{array} \right.
\end{equation*}
which yields
\begin{equation*}
\Vert u \Vert^2_{F^{2s}(T)} \lesssim \Vert u(0) \Vert_{H^{2s}}^2 + \Vert u \Vert_{F^{2s}(T)}^2 \Vert u \Vert_{F^s(T)} + \Vert u \Vert^2_{F^{2s}(T)} \Vert u \Vert_{F^s(T)}
\end{equation*}
This shows $\Vert u_2 \Vert_{F^{2s}(T)} \lesssim \Vert u_2(0) \Vert_{H^{2s}}$ provided that $\Vert u_2 \Vert_{F^s(T)}$ is sufficiently small. This implies 
\begin{equation*}
\Vert u \Vert_{F^{2s}(T)} \lesssim \Vert P_{\leq N} u \Vert_{H^{2s}} \lesssim N^s \Vert u \Vert_{H^s}
\end{equation*}
For the solution to the difference equation we derive the inequality
\begin{equation*}
\Vert v \Vert^2_{E^s(T)} \lesssim \Vert v(0) \Vert_{H^s}^2 + \Vert v \Vert^3_{F^s(T)} + \Vert v \Vert^2_{F^s(T)} \Vert u \Vert_{H^s}
\end{equation*}
This allows us to conclude a priori estimates for $v=S_T^\infty (u_0) - S_T^\infty(P_{\leq N} u_0)$.\\
Next, we consider a sequence of smooth initial data $(u_n) \subseteq H^\infty$ converging to $u_0 \in H^s, \; s > 1$.\\
We write
\begin{equation*}
\begin{split}
S_T^\infty(u_n) - S_T^\infty(u_m) &= (S_T^\infty(u_n) - S_T^\infty(P_{\leq N} u_n)) - (S_T^\infty(u_m) - S_T^\infty(P_{\leq N} u_m)) \\
 &+ (S_T^\infty(P_{\leq N} u_n) - S_T^\infty(P_{\leq N} u_m))
\end{split}
\end{equation*}
and by the above considerations
\begin{equation*}
\begin{split}
\Vert S_T^\infty(u_n) - S_T^\infty(u_m) \Vert_{C([0,T],H^s)} &\lesssim \Vert P_{\geq N} u_n \Vert_{H^s} + \Vert P_{\geq N} u_m \Vert_{H^s} \\
&+ \Vert S_T^\infty(P_{\leq N} u_n) - S_T^\infty(P_{\leq N} u_m) \Vert_{C([0,T],H^s)}
\end{split}
\end{equation*}
For the third term note that
\begin{equation*}
\begin{split}
\Vert S_T^\infty(P_{\leq N} u_n) - S_T^\infty(P_{\leq N} u_m) \Vert_{C([0,T],H^s)} &\leq \Vert S_T^\infty(P_{\leq N} u_n) - S_T^\infty(P_{\leq N} u_m) \Vert_{C([0,T],H^3)} \\
&\leq f(\Vert P_{\leq N} u_n - P_{\leq N} u_m \Vert_{H^3})
\end{split}
\end{equation*}
With $f(x) \rightarrow 0$ as $x \rightarrow 0$ due to continuous dependence in $H^3$. Since $\Vert P_{\leq N} u_n - P_{\leq N} u_m \Vert_{H^3} \lesssim N^{3-s} \Vert P_{\leq N}(u_n - u_m) \Vert_{H^s}$ we find that 
\begin{equation*}
\Vert S_T^\infty(P_{\leq N} u_n) - S_T^\infty(P_{\leq N} u_m) \Vert_{C([0,T],H^s)} \rightarrow 0 \text{ as } n,m \rightarrow \infty
\end{equation*}
 for any $N$. Choosing $N$ so that $\Vert P_{\geq N} u_n \Vert_{H^s} + \Vert P_{\geq N} u_m \Vert_{H^s} \leq \varepsilon/2$ for any $n,m$ which is possible due to convergence to $u$ we can complete the proof of Theorem \ref{thm:wellPosednessBenjaminOnoEquation} for $k=2$. When dealing with large initial data we rescale the initial value $u_0 \rightarrow \lambda u_0(\cdot / \lambda)$ to consider the Benjamin-Ono equation with small initial data on the rescaled torus $\lambda \mathbb{T}$. Following Remark \ref{remark:EuclideanTransversalityExtension} the decisive bilinear Strichartz estimate is scaling invariant, which allows us to rerun the above proof for small initial data on the rescaled torus. But note that this argument does not adapt in a simple manner to the case of higher order nonlinearities because of criticality or supercriticality of the $L^2$-norm.
\section{Extending the method to higher order nonlinearities}
\label{section:HigherOrderGeneralizations}
Finally, we indicate how the method can also deal with higher order nonlinearities. For definiteness consider the equation
\begin{equation}
\label{eq:kGeneralizedBenjaminOnoEquation}
\partial_t u + \mathcal{H} \partial_{xx} u = u^{k-1} \partial_x u, \quad u: \mathbb{R} \times \mathbb{T} \rightarrow \mathbb{R}
\end{equation}
For solutions we prove the following set of estimates:
\begin{equation}
\label{eq:GeneralizedBenjaminOnoEquationEstimates}
\left\{ \begin{array}{cl}
\Vert u \Vert_{F^s_{BO}(T)} &\lesssim \Vert u \Vert_{E^s_{BO}(T)} + \Vert \partial_x (u^k) \Vert_{N_{BO}^s(T)} \\
\Vert \partial_x (u^k) \Vert_{N^s(T)} &\lesssim \Vert u \Vert^k_{F_{BO}^s(T)} \\
\Vert u \Vert^2_{E^s(T)} &\lesssim \Vert u_0 \Vert^2_{H_{\mathbb{R}}^s} + \Vert u \Vert^{k+1}_{F^s(T)}
\end{array} \right.
\end{equation}
For the nonlinear interaction we prove the following estimates for $s>1$:
\begin{align}
\label{eq:NonlinearEstimateIGeneralizedBOEquation}
\Vert \partial_x (u_1 \ldots u_k) \Vert_{N^s_{BO}(T)} &\lesssim \prod_{l=1}^k \Vert u_k \Vert_{F^s_{BO}(T)} \\
\label{eq:NonlinearEstimateIIGeneralizedBOEquation}
\Vert \partial_x (u_1 \ldots u_k) \Vert_{N^0_{BO}(T)} &\lesssim \Vert u_1 \Vert_{F^0_{BO}(T)} \prod_{l=2}^{k} \Vert u_l \Vert_{F^s_{BO}(T)}
\end{align}
In case of $High \times \ldots \rightarrow High$-interaction the following crude estimate suffices:
\begin{equation}
\begin{split}
\Vert P_N (\partial_x (P_{N_1} u_1 \ldots P_{N_k} u_{k}) \Vert_{L_t^1 L_x^2} &\lesssim N^{1/2} \Vert P_{N_1} u_1 \ldots P_{N_k} u_k \Vert_{L_{t,x}^2} \\
&\lesssim \Vert P_{N_1} u_1 \Vert_{L_t^\infty L_x^2} \prod_{l=2}^k \Vert P_{N_l} u_l \Vert_{L_{t,x}^\infty} \\
&\lesssim \Vert P_{N_1} u_1 \Vert_{F^0_{BO}} \prod_{l=2}^k N_l^{1/2} \Vert P_{N_l} u_l \Vert_{F^0_{BO}}
\end{split}
\end{equation}
For $High \times High \times \ldots \rightarrow Low$-interaction observe
\begin{equation}
\begin{split}
&\Vert P_N (\partial_x (P_{N_1} u_1 P_{N_2} u_2 \ldots P_{N_k} u_k)) \Vert_{DU^2(I)} \\
&\lesssim N \frac{N_1}{N} \sup_{\Vert v \Vert_{V_0^2} = 1} \int_{\substack{I^\prime,\\ |I^\prime| = N_1^{-1}}} \int P_N v P_{N_1} u_1 P_{N_2} u_2 \ldots P_{N_k} u_k dx dt\\
&\lesssim N_1 \sup_{\Vert v \Vert_{V_0^2}=1} \Vert P_N v P_{N_1} u_1 \Vert_{L^2_{t,x}} \Vert P_{N_2} u_2 \Vert_{L_{t,x}^2} \prod_{l=3}^{k+1} \Vert P_{N_l} u_l \Vert_{L_{t,x}^\infty} \\
&\lesssim \log \langle \frac{N_1}{N} \rangle \Vert P_{N_1} u_1 \Vert_{F_{BO}} \Vert P_{N_2} u_2 \Vert_{F_{BO}} \prod_{l=3}^{k+1} N_l^{1/2} \Vert P_{N_l} u_l \Vert_{F_{BO}}
\end{split}
\end{equation}
which is again enough to conclude the estimates \eqref{eq:NonlinearEstimateIGeneralizedBOEquation}, \eqref{eq:NonlinearEstimateIIGeneralizedBOEquation}.\\
We turn to the modification of the energy estimate for solutions: After Littlewood-Paley decomposition, localization in time reciprocal to the highest frequencies and possibly integration by parts we have to estimate expressions of the kind
\begin{equation}
\label{eq:generalizedEnergyTransfer}
\int_{I,|I|=N_1^{-1}} \int P_{N_1} u P_{N_2} u P_{N_3} u \ldots P_{N_{k+1}} u dx dt,
\end{equation}
where $N_1 \sim N_2 \geq N_3 \geq \ldots \geq N_{k+1}$.\\
Suppose there are less than four frequencies, which are much larger than the rest. When this involves the input frequency, we integrate by parts to place the derivative on the $N_3$ frequency.\\
Next, we deploy two bilinear Strichartz estimates involving the largest to fourth to largest frequencies, which compensate for the factor $N_1$ from the time localization and use pointwise bounds for the remaining frequencies to find
\begin{equation}
\begin{split}
\eqref{eq:generalizedEnergyTransfer} &\lesssim \Vert P_{N_1} u P_{N_2} u \Vert_{L^2_{t,x}} \Vert P_{N_3} u P_{N_4} u \Vert_{L^2_{t,x}} \prod_{l=3}^{k+1} \Vert P_{N_l} u \Vert_{L^\infty_{t,x}} \\
&\lesssim N_1^{-1} \prod_{l=1}^4 \Vert P_{N_l} u \Vert_{F_{BO}} \prod_{l=5}^{k+1} \Vert P_{N_l} u \Vert_{L_{t,x}^\infty}
\end{split}
\end{equation}
and further
\begin{equation*}
\begin{split}
&\int_0^t ds \int dx P_{N_1} u \partial_x (P_{N_2} u P_{N_3} u P_{N_4} u \ldots P_{N_{k+1}} u) \\
&\lesssim N_3 \prod_{i=1}^4 \Vert P_{N_i} u \Vert_{F^0_{BO}} \prod_{l=5}^{k+1} N_l^{1/2} \Vert P_{N_l} u \Vert_{F^0_{BO}}
\end{split}
\end{equation*}
In case the input frequency is not involved the estimate can be deduced by the same means without integration by parts.\\
Next, suppose that $N_1 \sim N_2 \sim N_3 \sim N_4$. In this case we can use four $L^4_{t,x}$-Strichartz estimates on the high frequencies which gains a factor $N^{-1/2}$ due to H\"older in time and $L^{\infty}_{t,x}$-estimates on the remaining frequencies. Since we have two high frequencies to spare, this easily compensates for $N^{3/2}$.\\
We turn to differences of solutions $v=u_1-u_2$. Note the algebraic identities
\begin{equation}
\label{eq:factorizationDifferenceNonlinearities}
\partial_x (u_1^k - u_2^k) = \partial_x (v (P_k(u_1,u_2)) = \partial_x (v Q_k(v,u_2))
\end{equation}
First, we turn to the following set of estimates:
\begin{equation}
\label{eq:DifferencesSolutionsBenjaminOnoEquation}
\left\{ \begin{array}{cl}
\Vert v \Vert_{F^0_{BO}(T)} &\lesssim \Vert v \Vert_{E^0_{BO}(T)} + \Vert \partial_x (v P_k(u_1,u_2)) \Vert_{N_{BO}^0(T)} \\
\Vert \partial_x (v P_k(u_1,u_2) ) \Vert_{N^0(T)} &\lesssim \Vert v \Vert_{F^0_{BO}(T)} P_k( \Vert u_1 \Vert_{F_{BO}^s(T)}, \Vert u_2 \Vert_{F^s_{BO}(T)}) \\
\Vert u \Vert^2_{E^0(T)} &\lesssim \Vert v(0) \Vert^2_{L^2} + \Vert v \Vert_{F^0}^2 R_k ( \Vert u_1 \Vert_{F^s(T)}, \Vert u_2 \Vert_{F^s(T)})
\end{array} \right.
\end{equation}
for $s>1$, $T \in (0,1]$.\\
The nonlinear estimate is settled with \eqref{eq:NonlinearEstimateIIGeneralizedBOEquation}. For the energy estimate we in addition to above have to consider the interaction where we can no longer integrate by parts due to the reduced symmetry of the difference equation:
\begin{equation}
N_1^2 \int_{|I|=N_1^{-1}} \int P_{N_1} v P_{N_2} u_{i_2} P_{N_3} v P_{N_4} u_{i_4} \ldots dx dt,
\end{equation}
where $N_1 \sim N_2 \gg N_3 \geq N_4 \geq \ldots \geq N_{k+1}$.\\
We find
\begin{equation*}
\begin{split}
&\lesssim N_1^2 \Vert P_{N_1} v P_{N_3} v \Vert_{L^2_{t,x}} \Vert P_{N_2} u_{i_2} P_{N_4} u_{i_4} \Vert_{L^2_{t,x}} \prod_{l=5}^{k+1} \Vert P_{N_l} u_{i_l} \Vert_{L^\infty_{t,x}} \\
&\lesssim \Vert P_{N_1} v \Vert_{F_{BO}} \Vert P_{N_3} v \Vert_{F_{BO}} N_2 \Vert P_{N_2} u_{i_2} \Vert_{F_{BO}} \Vert P_{N_4} u_{i_4} \Vert_{F_{BO}} \prod_{l=5}^{k+1} N_l^{1/2} \Vert P_{N_l} u_{i_l} \Vert_{F_{BO}}
\end{split}
\end{equation*}
and the summation to conclude the energy estimate from \eqref{eq:DifferencesSolutionsBenjaminOnoEquation} is straight-forward.\\
Finally, we turn to the following set of estimates
\begin{equation*}
\left\{ \begin{array}{cl}
\Vert v \Vert_{F^s_{BO}(T)} &\lesssim \Vert v \Vert_{E^s_{BO}(T)} + \Vert \partial_x (v (P_k(u_1,u_2) )) \Vert_{N_{BO}^s(T)} \\
\Vert \partial_x (v (P_k(u_1,u_2) )) \Vert_{N^s(T)} &\lesssim \Vert v \Vert_{F_{BO}^s(T)} P_k(\Vert u_1 \Vert_{F^s_{BO}}, \Vert u_2 \Vert_{F^s_{BO}}) \\
\Vert v \Vert^2_{E^s(T)} &\lesssim \Vert v(0) \Vert^2_{H_{\mathbb{R}}^s} + \Vert v \Vert_{F^s_{BO}} R_k (\Vert v \Vert_{F^s_{BO}}, \Vert u_2 \Vert_{F^s_{BO}}) \\
&+ \Vert v \Vert_{F^0_{BO}} \Vert v \Vert_{F^s_{BO}} \Vert u_2 \Vert_{F^{2s}} L_k(\Vert v \Vert_{F^s_{BO}}, \Vert u_2 \Vert_{F^s_{BO}})
\end{array} \right.
\end{equation*}
The last term is new and comes from the estimate of
\begin{equation}
N_1 \int_{|I|=N_1^{-1}} \int P_{N_1} v P_{N_3} v P_{N_2} u_2 \prod_{l=4}^{l+1} P_{N_l} w_{i_l} dx dt,
\end{equation}
where $N_1 \sim N_2 \gg N_3 \geq \ldots \geq N_{k+1}$.\\
In this case we estimate $P_{N_2} u_2$ in $F^{2s}_{BO}$ which allows us to conclude the estimates. With the necessary estimates at our disposal we can prove local well-posedness for generalized Benjamin-Ono equations \eqref{eq:kGeneralizedBenjaminOnoEquation} like in Section \ref{section:Conclusion}.
\section*{Acknowledgement}
I would like to thank my thesis advisor Professor Sebastian Herr for helpful comments improving the presentation.

\bibliographystyle{amsxport}

\begin{thebibliography}{x}

\bib{BenArtziKochSaut2003}{article}{
      author={Ben-Artzi, Matania},
      author={Koch, Herbert},
      author={Saut, Jean-Claude},
       title={Dispersion estimates for third order equations in two
  dimensions},
        date={2003},
        ISSN={0360-5302},
     journal={Comm. Partial Differential Equations},
      volume={28},
      number={11-12},
       pages={1943\ndash 1974},
         url={https://doi.org/10.1081/PDE-120025491},
      review={\MR{2015408}},
}

\bib{BonaSmith1975}{article}{
      author={Bona, J.~L.},
      author={Smith, R.},
       title={The initial-value problem for the {K}orteweg-de {V}ries
  equation},
        date={1975},
        ISSN={0080-4614},
     journal={Philos. Trans. Roy. Soc. London Ser. A},
      volume={278},
      number={1287},
       pages={555\ndash 601},
         url={https://doi.org/10.1098/rsta.1975.0035},
      review={\MR{0385355}},
}

\bib{Bourgain1998}{article}{
      author={Bourgain, J.},
       title={Refinements of {S}trichartz' inequality and applications to
  {$2$}{D}-{NLS} with critical nonlinearity},
        date={1998},
        ISSN={1073-7928},
     journal={Internat. Math. Res. Notices},
      number={5},
       pages={253\ndash 283},
         url={https://doi.org/10.1155/S1073792898000191},
      review={\MR{1616917}},
}

\bib{BurqGerardTzvetkov2004}{article}{
      author={Burq, N.},
      author={G\'erard, P.},
      author={Tzvetkov, N.},
       title={Strichartz inequalities and the nonlinear {S}chr\"odinger
  equation on compact manifolds},
        date={2004},
        ISSN={0002-9327},
     journal={Amer. J. Math.},
      volume={126},
      number={3},
       pages={569\ndash 605},
  url={http://muse.jhu.edu/journals/american_journal_of_mathematics/v126/126.3burq.pdf},
      review={\MR{2058384}},
}

\bib{ChristHolmerTataru2012}{article}{
      author={Christ, Michael},
      author={Holmer, Justin},
      author={Tataru, Daniel},
       title={Low regularity a priori bounds for the modified {K}orteweg-de
  {V}ries equation},
        date={2012},
        ISSN={0278-5307},
     journal={Lib. Math. (N.S.)},
      volume={32},
      number={1},
       pages={51\ndash 75},
         url={http://dx.doi.org/10.14510/lm-ns.v32i1.32},
      review={\MR{3058496}},
}

\bib{GruenrockHerr2014}{article}{
      author={Gr\"{u}nrock, Axel},
      author={Herr, Sebastian},
       title={The {F}ourier restriction norm method for the
  {Z}akharov-{K}uznetsov equation},
        date={2014},
        ISSN={1078-0947},
     journal={Discrete Contin. Dyn. Syst.},
      volume={34},
      number={5},
       pages={2061\ndash 2068},
         url={https://doi.org/10.3934/dcds.2014.34.2061},
      review={\MR{3124726}},
}

\bib{GuoPengWangWang2011}{article}{
      author={Guo, Zihua},
      author={Peng, Lizhong},
      author={Wang, Baoxiang},
      author={Wang, Yuzhao},
       title={Uniform well-posedness and inviscid limit for the
  {B}enjamin-{O}no-{B}urgers equation},
        date={2011},
        ISSN={0001-8708},
     journal={Adv. Math.},
      volume={228},
      number={2},
       pages={647\ndash 677},
         url={https://doi.org/10.1016/j.aim.2011.03.017},
      review={\MR{2822208}},
}

\bib{HadacHerrKoch2009}{article}{
      author={Hadac, Martin},
      author={Herr, Sebastian},
      author={Koch, Herbert},
       title={Well-posedness and scattering for the {KP}-{II} equation in a
  critical space},
        date={2009},
        ISSN={0294-1449},
     journal={Ann. Inst. H. Poincar\'{e} Anal. Non Lin\'{e}aire},
      volume={26},
      number={3},
       pages={917\ndash 941},
         url={https://doi.org/10.1016/j.anihpc.2008.04.002},
      review={\MR{2526409}},
}

\bib{HadacHerrKoch2009Addendum}{article}{
      author={Hadac, Martin},
      author={Herr, Sebastian},
      author={Koch, Herbert},
       title={Erratum to ``{W}ell-posedness and scattering for the {KP}-{II}
  equation in a critical space'' [{A}nn. {I}. {H}. {P}oincar\'{e}---{AN} 26 (3)
  (2009) 917--941] [mr2526409]},
        date={2010},
        ISSN={0294-1449},
     journal={Ann. Inst. H. Poincar\'{e} Anal. Non Lin\'{e}aire},
      volume={27},
      number={3},
       pages={971\ndash 972},
         url={https://doi.org/10.1016/j.anihpc.2010.01.006},
      review={\MR{2629889}},
}

\bib{Hani2012}{article}{
      author={Hani, Zaher},
       title={A bilinear oscillatory integral estimate and bilinear refinements
  to {S}trichartz estimates on closed manifolds},
        date={2012},
        ISSN={2157-5045},
     journal={Anal. PDE},
      volume={5},
      number={2},
       pages={339\ndash 363},
         url={https://doi.org/10.2140/apde.2012.5.339},
      review={\MR{2970710}},
}

\bib{Herr2008}{article}{
      author={Herr, Sebastian},
       title={A note on bilinear estimates and regularity of flow maps for
  nonlinear dispersive equations},
        date={2008},
        ISSN={0002-9939},
     journal={Proc. Amer. Math. Soc.},
      volume={136},
      number={8},
       pages={2881\ndash 2886},
         url={https://doi.org/10.1090/S0002-9939-08-09238-1},
      review={\MR{2399054}},
}

\bib{IfrimTataru2017}{article}{
      author={{Ifrim}, M.},
      author={{Tataru}, D.},
       title={{Well-posedness and dispersive decay of small data solutions for
  the Benjamin-Ono equation}},
        date={2017-01},
     journal={ArXiv e-prints},
      eprint={1701.08476},
}

\bib{IonescuKenigTataru2008}{article}{
      author={Ionescu, A.~D.},
      author={Kenig, C.~E.},
      author={Tataru, D.},
       title={Global well-posedness of the {KP}-{I} initial-value problem in
  the energy space},
        date={2008},
        ISSN={0020-9910},
     journal={Invent. Math.},
      volume={173},
      number={2},
       pages={265\ndash 304},
         url={http://dx.doi.org/10.1007/s00222-008-0115-0},
      review={\MR{2415308}},
}

\bib{IonescuKenig2007}{article}{
      author={Ionescu, Alexandru~D.},
      author={Kenig, Carlos~E.},
       title={Global well-posedness of the {B}enjamin-{O}no equation in
  low-regularity spaces},
        date={2007},
        ISSN={0894-0347},
     journal={J. Amer. Math. Soc.},
      volume={20},
      number={3},
       pages={753\ndash 798},
         url={https://doi.org/10.1090/S0894-0347-06-00551-0},
      review={\MR{2291918}},
}

\bib{IonescuPausader2012}{article}{
      author={Ionescu, Alexandru~D.},
      author={Pausader, Benoit},
       title={The energy-critical defocusing {NLS} on {$\Bbb T^3$}},
        date={2012},
        ISSN={0012-7094},
     journal={Duke Math. J.},
      volume={161},
      number={8},
       pages={1581\ndash 1612},
         url={https://doi.org/10.1215/00127094-1593335},
      review={\MR{2931275}},
}

\bib{KochTzvetkov2003}{article}{
      author={Koch, H.},
      author={Tzvetkov, N.},
       title={On the local well-posedness of the {B}enjamin-{O}no equation in
  {$H^s({\Bbb R})$}},
        date={2003},
        ISSN={1073-7928},
     journal={Int. Math. Res. Not.},
      number={26},
       pages={1449\ndash 1464},
         url={https://doi.org/10.1155/S1073792803211260},
      review={\MR{1976047}},
}

\bib{KochTzvetkov2005}{article}{
      author={Koch, H.},
      author={Tzvetkov, N.},
       title={Nonlinear wave interactions for the {B}enjamin-{O}no equation},
        date={2005},
        ISSN={1073-7928},
     journal={Int. Math. Res. Not.},
      number={30},
       pages={1833\ndash 1847},
         url={https://doi.org/10.1155/IMRN.2005.1833},
      review={\MR{2172940}},
}

\bib{KochTataru2012}{article}{
      author={Koch, Herbert},
      author={Tataru, Daniel},
       title={Energy and local energy bounds for the 1-d cubic {NLS} equation
  in {$H^{-\frac14}$}},
        date={2012},
        ISSN={0294-1449},
     journal={Ann. Inst. H. Poincar\'e Anal. Non Lin\'eaire},
      volume={29},
      number={6},
       pages={955\ndash 988},
         url={http://dx.doi.org/10.1016/j.anihpc.2012.05.006},
      review={\MR{2995102}},
}

\bib{LinaresPantheeRobertTzvetkov2018}{article}{
      author={{Linares}, F.},
      author={{Panthee}, M.},
      author={{Robert}, T.},
      author={{Tzvetkov}, N.},
       title={{On the periodic Zakharov-Kuznetsov equation}},
        date={2018-09},
     journal={ArXiv e-prints},
      eprint={1809.02027},
}

\bib{MolinetPilod2015}{article}{
      author={Molinet, L.},
      author={Pilod, D.},
       title={Bilinear {S}trichartz {E}stimates for the {Z}akharov-{K}uznetsov
  equation and applications},
        date={2015},
        ISSN={0294-1449},
     journal={Ann. Inst. H. Poincar\'{e} Anal. Non Lin\'{e}aire},
      volume={32},
      number={2},
       pages={347\ndash 371},
      review={\MR{0385355}},
}

\bib{MolinetSautTzvetkov2001}{article}{
      author={Molinet, L.},
      author={Saut, J.~C.},
      author={Tzvetkov, N.},
       title={Ill-posedness issues for the {B}enjamin-{O}no and related
  equations},
        date={2001},
        ISSN={0036-1410},
     journal={SIAM J. Math. Anal.},
      volume={33},
      number={4},
       pages={982\ndash 988},
         url={https://doi.org/10.1137/S0036141001385307},
      review={\MR{1885293}},
}

\bib{Molinet2008}{article}{
      author={Molinet, Luc},
       title={Global well-posedness in {$L^2$} for the periodic
  {B}enjamin-{O}no equation},
        date={2008},
        ISSN={0002-9327},
     journal={Amer. J. Math.},
      volume={130},
      number={3},
       pages={635\ndash 683},
         url={https://doi.org/10.1353/ajm.0.0001},
      review={\MR{2418924}},
}

\bib{Molinet2013}{article}{
      author={Molinet, Luc},
       title={A note on the inviscid limit of the {B}enjamin-{O}no-{B}urgers
  equation in the energy space},
        date={2013},
        ISSN={0002-9939},
     journal={Proc. Amer. Math. Soc.},
      volume={141},
      number={8},
       pages={2793\ndash 2798},
         url={https://doi.org/10.1090/S0002-9939-2013-11693-X},
      review={\MR{3056569}},
}

\bib{MolinetPilod2012}{article}{
      author={Molinet, Luc},
      author={Pilod, Didier},
       title={The {C}auchy problem for the {B}enjamin-{O}no equation in {$L^2$}
  revisited},
        date={2012},
        ISSN={2157-5045},
     journal={Anal. PDE},
      volume={5},
      number={2},
       pages={365\ndash 395},
         url={https://doi.org/10.2140/apde.2012.5.365},
      review={\MR{2970711}},
}

\bib{MolinetRibaud2009}{article}{
      author={Molinet, Luc},
      author={Ribaud, Francis},
       title={Well-posedness in {$H^1$} for generalized {B}enjamin-{O}no
  equations on the circle},
        date={2009},
        ISSN={1078-0947},
     journal={Discrete Contin. Dyn. Syst.},
      volume={23},
      number={4},
       pages={1295\ndash 1311},
         url={https://doi.org/10.3934/dcds.2009.23.1295},
      review={\MR{2461852}},
}

\bib{MolinetVento2015}{article}{
      author={Molinet, Luc},
      author={Vento, St\'ephane},
       title={Improvement of the energy method for strongly nonresonant
  dispersive equations and applications},
        date={2015},
        ISSN={2157-5045},
     journal={Anal. PDE},
      volume={8},
      number={6},
       pages={1455\ndash 1495},
         url={http://dx.doi.org/10.2140/apde.2015.8.1455},
      review={\MR{3397003}},
}

\bib{MoyuaVega2008}{article}{
      author={Moyua, A.},
      author={Vega, L.},
       title={Bounds for the maximal function associated to periodic solutions
  of one-dimensional dispersive equations},
        date={2008},
        ISSN={0024-6093},
     journal={Bull. Lond. Math. Soc.},
      volume={40},
      number={1},
       pages={117\ndash 128},
         url={http://dx.doi.org/10.1112/blms/bdm096},
      review={\MR{2409184}},
}

\bib{StaffilaniTataru2002}{article}{
      author={Staffilani, Gigliola},
      author={Tataru, Daniel},
       title={Strichartz estimates for a {S}chr\"{o}dinger operator with
  nonsmooth coefficients},
        date={2002},
        ISSN={0360-5302},
     journal={Comm. Partial Differential Equations},
      volume={27},
      number={7-8},
       pages={1337\ndash 1372},
         url={https://doi.org/10.1081/PDE-120005841},
      review={\MR{1924470}},
}

\bib{Tao2004}{article}{
      author={Tao, Terence},
       title={Global well-posedness of the {B}enjamin-{O}no equation in
  {$H^1({\bf R})$}},
        date={2004},
        ISSN={0219-8916},
     journal={J. Hyperbolic Differ. Equ.},
      volume={1},
      number={1},
       pages={27\ndash 49},
         url={https://doi.org/10.1142/S0219891604000032},
      review={\MR{2052470}},
}

\bib{Wiener1979}{book}{
      author={Wiener, Norbert},
       title={Collected works with commentaries. {V}ol. {II}},
      series={Mathematicians of Our Time},
   publisher={MIT Press, Cambridge, Mass.-London},
        date={1979},
      volume={15},
        ISBN={0-262-23092-5},
        note={Generalized harmonic analysis and Tauberian theory; classical
  harmonic and complex analysis, Edited by Pesi Rustom Masani},
      review={\MR{554238}},
}

\end{thebibliography}
\end{document}